\documentclass{amsart}
\usepackage[margin=3cm]{geometry}
\usepackage{amsmath,amssymb,amsfonts}
\usepackage{amsthm}
\usepackage{mathrsfs}
\usepackage{mathtools}
\usepackage{xcolor}
\usepackage{soul}
\usepackage{cite}
\usepackage{lineno}
\usepackage{graphicx, epsfig}
\usepackage{subcaption}

\usepackage{booktabs}
\usepackage[hidelinks]{hyperref}
\usepackage{setspace}
\theoremstyle{plain}
\newtheorem{theorem}{Theorem}[section]
\newtheorem{proposition}[theorem]{Proposition}
\newtheorem{lemma}[theorem]{Lemma}
\newtheorem{corollary}[theorem]{Corollary}

\theoremstyle{definition}

\theoremstyle{remark}
\newtheorem{remark}[theorem]{Remark}

\graphicspath{{figures/}}
\makeatletter
\def\input@path{{sections/}}
\makeatother

\title{Shape Optimization for the Principal Eigenvalue of the Pucci Operator in Three Dimensions}

\author[Mohan Mallick]{Mohan Mallick}
\address{Department of Mathematics, Visvesvaraya National Institute of Technology, Nagpur, India}
\email{mohanmallick@mth.vnit.ac.in}
\author[Ram Baran Verma]{*Ram Baran Verma}
\address{Department of Mathematics, SRM University AP, Amaravati, India}
\email{rambaran.v@srmap.edu.in}

\date{Version 2.0 — \today}
\begin{document}

\begin{abstract}
We investigate shape optimization for the principal eigenvalue of the Pucci extremal operator
\[
\left\{
\begin{aligned}
-\mathcal{M}^+_{\lambda,\Lambda}(D^{2}u)&=\mu^{+}_{1}(\Omega)u &&\text{in }\Omega,\\
u &=0 &&\text{on }\partial\Omega,
\end{aligned}
\right.
\]
in dimension three. Since $\mathcal{M}^+_{\lambda,\Lambda}$ is fully nonlinear, in non-divergence form, and non-variational, classical symmetrization and rearrangement methods do not apply.

We introduce a three-dimensional family of double--pyramidal domains $\{\Omega^\omega_{\gamma,a}\}$ parametrized by an anisotropy factor
\(
\gamma\in
\left[
\frac{2}{\sqrt\omega},
\frac{\sqrt\omega}{2}
\right]
\)
and an affine shear parameter \(a\in(-\pi,\pi)\), under the ellipticity
condition
\(
\omega=\frac{\Lambda}{\lambda}\geq 4.
\) Within this family and under a fixed-volume constraint, we prove that the scale-invariant principal eigenvalue is uniquely minimized at the symmetric unsheared configuration $(\gamma,a)=(1,0)$ among domains in the family $\{\Omega^\omega_{\gamma,a}\}$.

The proof combines an explicit construction of positive eigenfunctions on seven patches with a lower bound under affine shear deformations. Using the homogeneity and orthogonal invariance of the Pucci operator, we identify an involutive symmetry $\gamma\mapsto \gamma^{-1}$ in the associated volume functional and establish strict monotonicity away from the self-dual point $\gamma=1$. In particular, for $\omega\geq 4$, any nontrivial anisotropy or shear strictly increases the normalized principal eigenvalue.

This reveals a genuinely three-dimensional rigidity mechanism for a fully nonlinear spectral problem and extends to dimension three the symmetry-minimization phenomenon previously known in the planar case.
\end{abstract}

\maketitle



\vspace{0.2cm}

\section{Introduction}

The interaction between the geometry of a domain and the principal eigenvalue of elliptic operators lies at the heart of spectral geometry. A central problem in this direction is the shape optimization question:

\medskip
\emph{Which domains of fixed volume minimize the first Dirichlet eigenvalue?}
\medskip

For the Laplace operator,
\[
\left\{
\begin{aligned}
-\Delta u &= \mu_1(\Omega) u &&\text{in } \Omega,\\
u &> 0 &&\text{in } \Omega,\\
u &= 0 &&\text{on } \partial \Omega,
\end{aligned}
\right.
\]
the celebrated Rayleigh–Faber–Krahn inequality asserts that among all bounded open sets $\Omega \subset \mathbb{R}^n$ of prescribed volume, the Euclidean ball uniquely minimizes $\mu_1(\Omega)$ \cite{Rayleigh1877,Henrot2006,Faber1923,Krahn1925}.  

A related conjecture of P\'olya and Szeg\H{o} states that among all $n$--gons of fixed area, the regular $n$--gon minimizes the first eigenvalue \cite{PolyaSzego1951}. This is known for $n=3,4$ via Steiner symmetrization; see \cite{PolyaSzego1951,BogoselBucur2022}. The case $n\ge 5$ remains open in full generality, although substantial progress has recently been made; see \cite{BogoselBucur2022}.

Two structural features underlie these classical results:
\begin{enumerate}
\item the variational characterization of $\mu_1(\Omega)$ through the Rayleigh quotient;
\item the compatibility of the Laplacian with symmetrization and rearrangement techniques.
\end{enumerate}
In this setting, symmetry-minimization emerges naturally from energy considerations.

\medskip

For fully nonlinear elliptic operators in non-divergence form, however, neither variational structure nor symmetrization methods are available. In this work, we study the principal eigenvalue associated with the Pucci extremal operator
\begin{equation}\label{eigen}
\left\{
\begin{aligned}
-\mathcal{M}^+_{\lambda,\Lambda}(D^2 u)
&= \mu_1^+(\Omega) u &&\text{in } \Omega,\\
u &> 0 &&\text{in } \Omega,\\
u &= 0 &&\text{on } \partial \Omega,
\end{aligned}
\right.
\end{equation}
where $0<\lambda\le\Lambda$ and $\omega=\Lambda/\lambda\ge1$.
The operator $\mathcal{M}^+_{\lambda,\Lambda}$ is fully nonlinear, positively homogeneous of degree one, and non-variational. The principal half–eigenvalues $\mu_1^+(\Omega)$ and $\mu_1^-(\Omega)$ are defined in a viscosity sense \cite{QuaasSirakov2008} via the maximum principle framework of Berestycki–Nirenberg–Varadhan \cite{BerestyckiNirenbergVaradhan1994}; in general $\mu_1^+(\Omega) < \mu_1^-(\Omega)$ when $\lambda\neq\Lambda$.

The absence of a Rayleigh quotient makes symmetry-minimization for $\mu_1^+$ a genuinely structural question. Even the basic separation-of-variables mechanism fails: eigenfunctions of $\mathcal{M}^+_{\lambda,\Lambda}$ on cubes are not separable unless $\lambda=\Lambda$ (see Theorem~\ref{thm:nonsep}). Thus geometry must replace variational structure.

\medskip

The first symmetry-minimization result in this non-variational framework was obtained by Birindelli--Leoni \cite{BirindelliLeoni2014} in dimension two. For a family of planar domains parametrized by an asymmetry factor $\gamma \in \left[\frac{1}{\sqrt{\omega}},\sqrt{\omega}\right]$, they constructed explicit positive eigenfunctions and proved that, at fixed area, the principal eigenvalue is minimized in the symmetric configuration $\gamma=1$. 


\medskip

Their construction uses a geometric feature specific to the planar setting. In the isotropic case (\(\omega=1\)), the admissible planar domain is a rotated square, and hence one recovers a separable Laplacian configuration after rotation.

In dimension three, this simplification is no longer available. The symmetric geometry has a genuinely double-pyramidal character and cannot be reduced, by an orthogonal change of variables, to a cubical separable configuration. Moreover, since the Pucci operator is invariant under orthogonal transformations but not under general linear changes of variables, an arbitrary affine reduction is not compatible with the operator. Thus the three-dimensional extension is not a routine dimensional analogue of the planar construction. The geometry must be designed so that the sign structure of the Hessian eigenvalues can be controlled patch by patch while preserving \(C^1\) compatibility across interfaces.

\medskip


Even for the Laplacian, symmetry-minimization phenomena in dimension three are much less understood in polyhedral classes. While the ball minimizes the first Dirichlet eigenvalue among all domains of fixed volume \cite{Faber1923,Krahn1925,Henrot2006}, analogues of the P\'0lya--Szeg\H{o} conjecture for restricted polygonal or polyhedral families require separate arguments. Symmetry is immediate only in special classes, such as axis-aligned boxes, where the cube minimizes the first Dirichlet eigenvalue by separation of variables and the arithmetic--geometric mean inequality.

Against this backdrop, establishing symmetry-minimization for a fully nonlinear operator in three dimensions requires a different rigidity mechanism.

\medskip

\noindent\textbf{Main Result.}
We introduce a three-dimensional family of double--pyramidal domains $\Omega^\omega_{\gamma,a}\subset\mathbb{R}^3$, parametrized by:
\begin{enumerate}
\item an anisotropy parameter $\gamma\in\left[\frac{2}{\sqrt{\omega}},\frac{\sqrt{\omega}}{2}\right]$ controlling stretching,
\item a shear parameter $a\in(-\pi,\pi)$ arising from affine deformation,
\item the ellipticity ratio $\omega=\Lambda/\lambda\geq 4$.
\end{enumerate}

We emphasize that the optimization problem considered here is within the natural two-parameter family $\{\Omega^\omega_{\gamma,a}\}$, rather than over all domains in $\mathbb{R}^3$.

\begin{theorem}[Symmetry and shear rigidity]
Fix \(\omega\geq 4\) and a target volume \(V>0\). For each admissible pair
\[
\gamma\in
\left[
\frac{2}{\sqrt{\omega}},
\frac{\sqrt{\omega}}{2}
\right],
\qquad
a\in(-\pi,\pi),
\]
let \(\widetilde{\Omega}^{\omega}_{\gamma,a}\) denote the scaling of
\(\Omega^{\omega}_{\gamma,a}\) to volume \(V\). Then the principal positive
half-eigenvalue
\[
\mu_1^+(\widetilde{\Omega}^{\omega}_{\gamma,a})
\]
is uniquely minimized at the symmetric unsheared configuration
\((\gamma,a)=(1,0)\).

Equivalently, the scale-invariant quantity
\[
|\Omega^{\omega}_{\gamma,a}|^{2/3}
\mu_1^+(\Omega^{\omega}_{\gamma,a})
\]
is uniquely minimized at \((\gamma,a)=(1,0)\).
\end{theorem}




\medskip

The proof rests on four structural ingredients:
\begin{enumerate}
\item an explicit construction of positive eigenfunctions on a seven–patch partition of a fundamental octant;
\item a $C^1$ gluing mechanism ensuring global viscosity solutions;
\item a sharp spectral lower bound under affine shear deformations;
\item the symmetry $\gamma\mapsto\gamma^{-1}$ in the associated volume functional, yielding strict monotonicity away from $\gamma=1$.
\end{enumerate}

Together, these mechanisms show a three-dimensional rigidity phenomenon:
within the admissible range \(\omega\geq4\), any nontrivial anisotropic
stretching or shear strictly increases the normalized principal eigenvalue.

\medskip
\noindent\textbf{Organization of the paper.}
Section 2 recalls preliminaries on Pucci operators and principal half-eigenvalues.
Section 3 introduces the admissible double--pyramidal domains.
Section 4 constructs the explicit eigenfunctions and proves the identification of the principal eigenpair on the unsheared domains.
Section 5 establishes the spectral lower bound under shear and proves the minimization result within the family.
Technical Hessian computations and the volume monotonicity argument are collected in Appendix~A.

\section{Preliminaries}\label{sec:preliminaries}

In this section we recall the structural properties of the Pucci extremal operator and the principal half-eigenvalues that will be used throughout the paper.

\subsection{Pucci Extremal Operator}

Let $0<\lambda\le\Lambda$ and denote by $e_1,\dots,e_n$ the eigenvalues of a symmetric matrix $X\in\mathcal{S}^n$. The Pucci extremal operators are defined by
\begin{equation}\label{pucci}
\mathcal{M}^+_{\lambda,\Lambda}(X)
= \Lambda \sum_{e_i>0} e_i + \lambda \sum_{e_i<0} e_i,
\qquad
\mathcal{M}^-_{\lambda,\Lambda}(X)
= \lambda \sum_{e_i>0} e_i + \Lambda \sum_{e_i<0} e_i.
\end{equation}

The operator $\mathcal{M}^+_{\lambda,\Lambda}$ is fully nonlinear, uniformly elliptic, and positively homogeneous of degree one. We denote the ellipticity ratio by
\[
\omega = \frac{\Lambda}{\lambda} \ge 1.
\]
If $\lambda=\Lambda$, then $\mathcal{M}^+_{\lambda,\Lambda}(X) = \lambda tr(X)$, and hence, when $X=D^2u$, $\mathcal{M}^+_{\lambda,\Lambda}(D^2u)=\lambda\Delta u$.

\subsection{Structural Obstruction: Non-separability}

A fundamental difference with the Laplacian arises when $\Lambda>\lambda$: separation of variables fails even on Cartesian domains. 
\begin{theorem}[Non-separability on cubes]\label{thm:nonsep}
Let $Q=\left(-l,l\right)^n\subset\mathbb{R}^n$ be a cube of side length $2l$ and assume $\Lambda>\lambda>0$. Any positive eigenfunction of $\mathcal{M}^+_{\lambda,\Lambda}$ associated with $\mu_1^+(Q)$ cannot be written as a product of one-variable functions $u(x)=\displaystyle{\prod_{j=1}^{n}f(x_j)}$.
\end{theorem}

\noindent
The proof is deferred to Appendix~\ref{appendix:nonsep}. This obstruction forces us to abandon product structures and instead design a geometry compatible with piecewise sign control of Hessian eigenvalues.

\subsection{Principal Half-Eigenvalues}

Since $\mathcal{M}^+_{\lambda,\Lambda}$ is non-variational, the Rayleigh quotient characterization is not available. Following Berestycki–Nirenberg–Varadhan \cite{BerestyckiNirenbergVaradhan1994}, the principal positive half-eigenvalue of $\mathcal{M}^+_{\lambda,\Lambda}$ in a bounded domain $\Omega$ is defined (in the viscosity sense) by
\begin{equation}\label{eq:principal-positive}
\mu_1^+(\Omega)
=
\sup\left\{
\mu\in\mathbb{R} :
\exists\, \phi>0 \text{ in } \Omega,
\ \mathcal{M}^+_{\lambda,\Lambda}(D^2\phi)+\mu\phi\le0
\text{ in } \Omega
\right\}.
\end{equation}

Similarly, the principal negative half-eigenvalue is
\begin{equation}\label{eq:principal-negative}
\mu_1^-(\Omega)
=
\sup\left\{
\mu\in\mathbb{R} :
\exists\, \phi<0 \text{ in } \Omega,
\ \mathcal{M}^+_{\lambda,\Lambda}(D^2\phi)+\mu\phi\ge0
\text{ in } \Omega
\right\}.
\end{equation}

It is known that $\mu_1^+(\Omega)>0$, that a corresponding positive eigenfunction exists and is unique up to positive scalar multiplication, and that $\mu_1^+(\Omega)<\mu_1^-(\Omega)$ when $\lambda\neq\Lambda$; see \cite{QuaasSirakov2008}.

The positive homogeneity of $\mathcal{M}^+_{\lambda,\Lambda}$ implies the scaling law
\begin{equation}\label{eq:scaling}
\mu_1^\pm(t\Omega)=t^{-2}\mu_1^\pm(\Omega),
\qquad t>0,
\end{equation}
which will be used repeatedly in the normalization arguments.

\subsection{A Gluing Principle}

The construction of eigenfunctions relies on patching local viscosity solutions across smooth interfaces.

\begin{lemma}[Gluing principle]\label{lem:gluing}
Let \(\Omega_1,\Omega_2\subset\mathbb R^n\) be open sets with a common
\(C^2\) interface
\[
\Gamma:=\partial\Omega_1\cap\partial\Omega_2,
\]
and assume that
\[
\Omega:=\Omega_1\cup\Omega_2\cup\Gamma
\]
is open. Suppose that
\[
u_i\in C^2(\Omega_i)\cap C^1(\Omega_i\cup\Gamma),
\qquad i=1,2,
\]
satisfy
\[
-\mathcal M^+_{\lambda,\Lambda}(D^2u_i)=\mu u_i
\quad\text{in }\Omega_i,
\]
and assume the compatibility conditions
\[
u_1=u_2,\qquad Du_1=Du_2
\quad\text{on }\Gamma.
\]
Define
\[
u(x)=
\begin{cases}
u_1(x), & x\in\Omega_1,\\
u_2(x), & x\in\Omega_2,\\
u_1(x)=u_2(x), & x\in\Gamma.
\end{cases}
\]
Then \(u\in C^1(\Omega)\) and \(u\) is a viscosity solution of
\[
-\mathcal M^+_{\lambda,\Lambda}(D^2u)=\mu u
\quad\text{in }\Omega.
\]
\end{lemma}

\begin{proof}
The $C^1$ matching assumptions imply that the piecewise definition
\[
u(x)=
\begin{cases}
u_1(x), & x\in \Omega_1,\\
u_2(x), & x\in \Omega_2,\\
u_1(x)=u_2(x), & x\in \Gamma,
\end{cases}
\]
is well posed and yields a function $u\in C^1(\Omega)$.

Set
\[
F(X,r):=-\mathcal{M}^+_{\lambda,\Lambda}(X)-\mu r.
\]
We show that the equation
\(
F(D^2u,u)=0
\)
holds in the viscosity sense. Away from the interface, $u$ agrees locally with either $u_1$ or $u_2$, and since each $u_i$ is a classical solution in $\Omega_i$, the viscosity property is automatic there. It therefore remains to check the equation at points $x_0\in \Gamma$.



We prove the subsolution property; the supersolution argument is analogous.
Let $\varphi\in C^2(\Omega)$ be such that $u-\varphi$ attains a local maximum
at $x_0\in\Gamma$. We must show that
\[
-\mathcal{M}^+_{\lambda,\Lambda}(D^2\varphi(x_0))
\le \mu\,u(x_0).
\]
Since $u\in C^1(\Omega)$, necessarily
\[
D\varphi(x_0)=Du(x_0).
\]
By replacing $\varphi$ with $\varphi+|x-x_0|^4$, we may assume that
$x_0$ is a strict local maximum of $u-\varphi$.

Since $\Gamma$ is a $C^2$ hypersurface, there exists a signed distance
function $d\in C^2(U)$ in a neighborhood $U$ of $\Gamma$ such that
\[
d=0 \text{ on }\Gamma,\qquad d>0 \text{ in }\Omega_1,\qquad d<0 \text{ in }\Omega_2.
\]
For $\varepsilon>0$, define
\[
\varphi_\varepsilon(x)
:=\varphi(x)-\varepsilon d(x)+\sqrt{\varepsilon}\,|x-x_0|^2.
\]

We claim that, for all sufficiently small \(\varepsilon>0\), the function
\(u_1-\varphi_\varepsilon\) attains a local maximum at some point
\(x_\varepsilon\in\Omega_1\), with \(x_\varepsilon\to x_0\).

Indeed, on \(\Gamma\) we have \(d=0\) and \(u_1=u\), hence
\[
u_1-\varphi_\varepsilon
=
u-\varphi-\sqrt{\varepsilon}|x-x_0|^2\leq 0,
\]
with equality only at \(x_0\). On the other hand, since
\(D(u_1-\varphi)(x_0)=0\) and \(d(x_0+t\nu)=t\) for the unit normal
\(\nu\) pointing into \(\Omega_1\), we may choose \(t_\varepsilon>0\),
with \(t_\varepsilon\to0\), such that, for
\(y_\varepsilon=x_0+t_\varepsilon\nu\),
\[
(u_1-\varphi_\varepsilon)(y_\varepsilon)>0.
\]
Therefore the local maximum of \(u_1-\varphi_\varepsilon\) near \(x_0\)
cannot occur on \(\Gamma\). Hence it is attained at some interior point
\(x_\varepsilon\in\Omega_1\). Since \(x_0\) is a strict local maximum of
\(u-\varphi\), we also have \(x_\varepsilon\to x_0\).

Now $u_1-\varphi_\varepsilon$ has a local maximum at the interior point
$x_\varepsilon$, and $u_1$ is a classical solution in $\Omega_1$. Hence
\[
-\mathcal{M}^+_{\lambda,\Lambda}(D^2\varphi_\varepsilon(x_\varepsilon))
\le \mu\,u_1(x_\varepsilon).
\]
Since
\[
D^2\varphi_\varepsilon
=
D^2\varphi-\varepsilon D^2d+2\sqrt{\varepsilon}\,I,
\]
we have
\[
D^2\varphi_\varepsilon(x_\varepsilon)\to D^2\varphi(x_0),
\qquad
u_1(x_\varepsilon)\to u_1(x_0)=u(x_0).
\]
Passing to the limit and using continuity of $\mathcal{M}^+_{\lambda,\Lambda}$,
we obtain
\[
-\mathcal{M}^+_{\lambda,\Lambda}(D^2\varphi(x_0))
\le \mu\,u(x_0).
\]
Thus $u$ is a viscosity subsolution at $x_0$.

The supersolution property is proved in the same way, starting from a test
function $\varphi$ such that $u-\varphi$ has a local minimum at $x_0$, and
using the perturbation
\[
\psi_\varepsilon(x)
:=\varphi(x)+\varepsilon d(x)-\sqrt{\varepsilon}\,|x-x_0|^2.
\]
This yields
\[
-\mathcal{M}^+_{\lambda,\Lambda}(D^2\varphi(x_0))
\ge \mu\,u(x_0).
\]

Therefore $u$ is both a viscosity subsolution and a viscosity supersolution
at every point of $\Gamma$, and hence a viscosity solution of
\[
-\mathcal{M}^+_{\lambda,\Lambda}(D^2u)=\mu\,u
\qquad\text{in }\Omega.
\]
\end{proof}

\begin{remark}
Although Lemma~\ref{lem:gluing} is stated for two subdomains separated by one smooth interface, we shall use it below for a finite patchwise decomposition. This causes no additional difficulty. Indeed, the lemma can be applied successively across the finitely many interfaces, since the local pieces are \(C^2\) inside each patch and agree up to first order on every common interface. At lower-dimensional junctions, the same viscosity test-function argument applies locally, because the globally defined function is \(C^1\) and satisfies the equation classically on each adjacent patch. Thus a finite \(C^1\)-compatible patchwise construction gives a global viscosity solution.
\end{remark}
\section{Admissible Domains: A Double--Pyramidal Geometry in $\mathbb{R}^3$}
\label{sec:construction}

Motivated by the failure of separation of variables for
\(\mathcal M^+_{\lambda,\Lambda}\) on cubes (Theorem~\ref{thm:nonsep}),
we introduce a three-dimensional family of domains adapted to the
sign structure of the Hessian.

The construction is inspired by the two-dimensional rhombus geometry
of Birindelli--Leoni~\cite{BirindelliLeoni2014}, but the
three-dimensional case requires a different construction. The domains
are double-pyramidal-type domains obtained by attaching curved caps
and bridge regions to a central cube.

A key feature is that the portion of the domain in the first octant
decomposes into seven regions. This seven-patch structure allows us to
separate the different Hessian inertia patterns that appear in the
piecewise eigenfunction constructed in Section~\ref{sec:eigenfunctions},
while keeping the interfaces explicit.

Let \(0<\lambda\leq\Lambda\), and assume that
\[
\omega:=\frac{\Lambda}{\lambda}\geq4.
\]
We set
\[
\Gamma_\omega:=
\left[
\frac{2}{\sqrt\omega},
\frac{\sqrt\omega}{2}
\right],
\qquad
\alpha:=\frac{\pi}{2}.
\]
For the remainder of this section, fix \(\gamma\in\Gamma_\omega\). Notice that
\[
\gamma\in\Gamma_\omega
\quad\Longleftrightarrow\quad
\gamma^{-1}\in\Gamma_\omega,
\]
which gives the symmetry \(\gamma\mapsto\gamma^{-1}\) used later in the
volume monotonicity argument.

The restriction \(\omega\geq4\) and the above choice of
\(\Gamma_\omega\) ensure that the face caps attach to the whole
corresponding faces of the central cube. Indeed, for the \(Z\)-cap one has
\[
0\leq
\frac{\gamma}{\sqrt\omega}(\cos x+\cos y)
\leq
\frac{2\gamma}{\sqrt\omega}
\leq1,
\]
while for the \(X\)- and \(Y\)-caps,
\[
0\leq
\frac{\gamma\cos y+\cos z}{\gamma\sqrt\omega}
\leq
\frac{1+\gamma^{-1}}{\sqrt\omega}
\leq1.
\]
Thus the inverse trigonometric expressions defining the face caps are
well defined on the full cube faces.

We stress that the definitions below are purely geometric. The associated
piecewise eigenfunction will be introduced only in
Section~\ref{sec:eigenfunctions}. All inverse trigonometric functions are
taken on their principal branches. In the bridge regions, the displayed
conditions of the form
\[
0\leq \cdots \leq 1
\]
ensure that the corresponding \(\arccos\) expressions are well defined.

\subsection{Construction in the first octant}

Let
\[
\mathbb{R}^3_{\ge 0}:=\{(x,y,z)\in\mathbb{R}^3:x,y,z\ge 0\}.
\]
The portion of the domain in the first octant is obtained by adjoining three face caps and three edge bridges to a central cube.

\paragraph{Central region (P1).}
\[
\Omega^{\omega,+}_{\gamma}[\mathrm{C}]
:=
\big\{(x,y,z)\in\mathbb{R}^3_{\ge 0}:0\le x,y,z\le \alpha\big\}.
\]

\paragraph{Face caps (P2--P4).}
\begin{align*}
\Omega^{\omega,+}_{\gamma}[\mathrm{Z}]
&:=
\Big\{(x,y,z)\in\mathbb{R}^3_{\ge 0}:
0\le x,y\le \alpha,\;
\alpha\le z\le \alpha+\sqrt{\omega}\,\arcsin\!\Big(\frac{\gamma}{\sqrt{\omega}}(\cos x+\cos y)\Big)\Big\},\\[4pt]
\Omega^{\omega,+}_{\gamma}[\mathrm{X}]
&:=
\Big\{(x,y,z)\in\mathbb{R}^3_{\ge 0}:
0\le y,z\le \alpha,\;
\alpha\le x\le \alpha+\sqrt{\omega}\,\arcsin\!\Big(\frac{\gamma\cos y+\cos z}{\gamma\sqrt{\omega}}\Big)\Big\},\\[4pt]
\Omega^{\omega,+}_{\gamma}[\mathrm{Y}]
&:=
\Big\{(x,y,z)\in\mathbb{R}^3_{\ge 0}:
0\le x,z\le \alpha,\;
\alpha\le y\le \alpha+\sqrt{\omega}\,\arcsin\!\Big(\frac{\gamma\cos x+\cos z}{\gamma\sqrt{\omega}}\Big)\Big\}.
\end{align*}

\paragraph{Edge bridges (P5--P7).}
The sharp coordinate extents of the bridge pieces are determined by the traces of the
adjacent cap regions. Thus
\begin{align*}
\Omega^{\omega,+}_{\gamma}&[\mathrm{ZX}]
:=
\Big\{(x,y,z)\in\mathbb{R}^3_{\ge 0}:
\alpha\le x\le \alpha+\sqrt{\omega}\arcsin\!\Big(\frac{1}{\sqrt{\omega}}\Big),\;
\alpha\le z\le \alpha+\sqrt{\omega}\arcsin\!\Big(\frac{\gamma}{\sqrt{\omega}}\Big),\\
&0\le \sqrt{\omega}\sin\!\Big(\frac{x-\alpha}{\sqrt{\omega}}\Big)
+\frac{\sqrt{\omega}}{\gamma}\sin\!\Big(\frac{z-\alpha}{\sqrt{\omega}}\Big)\le 1,\;
0\le y\le \arccos\!\Big(\sqrt{\omega}\sin\!\Big(\frac{x-\alpha}{\sqrt{\omega}}\Big)
+\frac{\sqrt{\omega}}{\gamma}\sin\!\Big(\frac{z-\alpha}{\sqrt{\omega}}\Big)\Big)\Big\},\\[8pt]
\Omega^{\omega,+}_{\gamma}&[\mathrm{XY}]
:=
\Big\{(x,y,z)\in\mathbb{R}^3_{\ge 0}:
\alpha\le x\le \alpha+\sqrt{\omega}\arcsin\!\Big(\frac{1}{\gamma\sqrt{\omega}}\Big),\;
\alpha\le y\le \alpha+\sqrt{\omega}\arcsin\!\Big(\frac{1}{\gamma\sqrt{\omega}}\Big),\\
&0\le \gamma\sqrt{\omega}\sin\!\Big(\frac{x-\alpha}{\sqrt{\omega}}\Big)
+\gamma\sqrt{\omega}\sin\!\Big(\frac{y-\alpha}{\sqrt{\omega}}\Big)\le 1,\;
0\le z\le \arccos\!\Big(\gamma\sqrt{\omega}\sin\!\Big(\frac{x-\alpha}{\sqrt{\omega}}\Big)
+\gamma\sqrt{\omega}\sin\!\Big(\frac{y-\alpha}{\sqrt{\omega}}\Big)\Big)\Big\},\\[8pt]
\Omega^{\omega,+}_{\gamma}&[\mathrm{YZ}]
:=
\Big\{(x,y,z)\in\mathbb{R}^3_{\ge 0}:
\alpha\le y\le \alpha+\sqrt{\omega}\arcsin\!\Big(\frac{1}{\sqrt{\omega}}\Big),\;
\alpha\le z\le \alpha+\sqrt{\omega}\arcsin\!\Big(\frac{\gamma}{\sqrt{\omega}}\Big),\\
&0\le \sqrt{\omega}\sin\!\Big(\frac{y-\alpha}{\sqrt{\omega}}\Big)
+\frac{\sqrt{\omega}}{\gamma}\sin\!\Big(\frac{z-\alpha}{\sqrt{\omega}}\Big)\le 1,\;
0\le x\le \arccos\!\Big(\sqrt{\omega}\sin\!\Big(\frac{y-\alpha}{\sqrt{\omega}}\Big)
+\frac{\sqrt{\omega}}{\gamma}\sin\!\Big(\frac{z-\alpha}{\sqrt{\omega}}\Big)\Big)\Big\}.
\end{align*}

We define the first-octant domain by
\[
\Omega^{\omega,+}_{\gamma}
:=
\Omega^{\omega,+}_{\gamma}[\mathrm{C}]
\cup \Omega^{\omega,+}_{\gamma}[\mathrm{Z}]
\cup \Omega^{\omega,+}_{\gamma}[\mathrm{X}]
\cup \Omega^{\omega,+}_{\gamma}[\mathrm{Y}]
\cup \Omega^{\omega,+}_{\gamma}[\mathrm{ZX}]
\cup \Omega^{\omega,+}_{\gamma}[\mathrm{XY}]
\cup \Omega^{\omega,+}_{\gamma}[\mathrm{YZ}].
\]
Here and below, the displayed non-strict inequalities describe the closures of the patch pieces; the domain itself is understood as the interior of the resulting union.

\subsection{Full domain by reflection}

The full domain is obtained by symmetry across the coordinate planes:
\[
\Omega^{\omega}_{\gamma}
:=
\big\{(\sigma_1x,\sigma_2y,\sigma_3z):(x,y,z)\in\Omega^{\omega,+}_{\gamma},\; \sigma_1,\sigma_2,\sigma_3\in\{-1,1\}\big\}.
\]


Geometrically, \(\Omega^\omega_\gamma\) is a bounded double--pyramidal-type domain with a central cubical core, curved caps, and bridge regions arranged symmetrically with respect to the coordinate planes. The explicit piecewise eigenfunction associated with this geometry will be introduced in the next section, where we also verify the \(C^1\) compatibility across the patch interfaces.

For illustration, Figure~\ref{fig:three_domains} shows three representative
domains corresponding to \(\lambda=1\) and \(\Lambda=16\), so that
\(\omega=16\). The left panel shows the first-octant domain
\(\Omega^{\omega,+}*{\gamma}\) in the symmetric case \(\gamma=1\), together
with the seven-patch construction. The middle panel shows the corresponding
full domain \(\Omega^\omega*{\gamma}\) obtained by reflection across the
coordinate planes. The right panel shows a nonsymmetric example, with
\(\gamma=1.5\), illustrating the effect of anisotropy on the geometry.

\begin{figure}[htbp]
\centering

\begin{subfigure}[b]{0.32\textwidth}
\centering
\includegraphics[width=\textwidth]{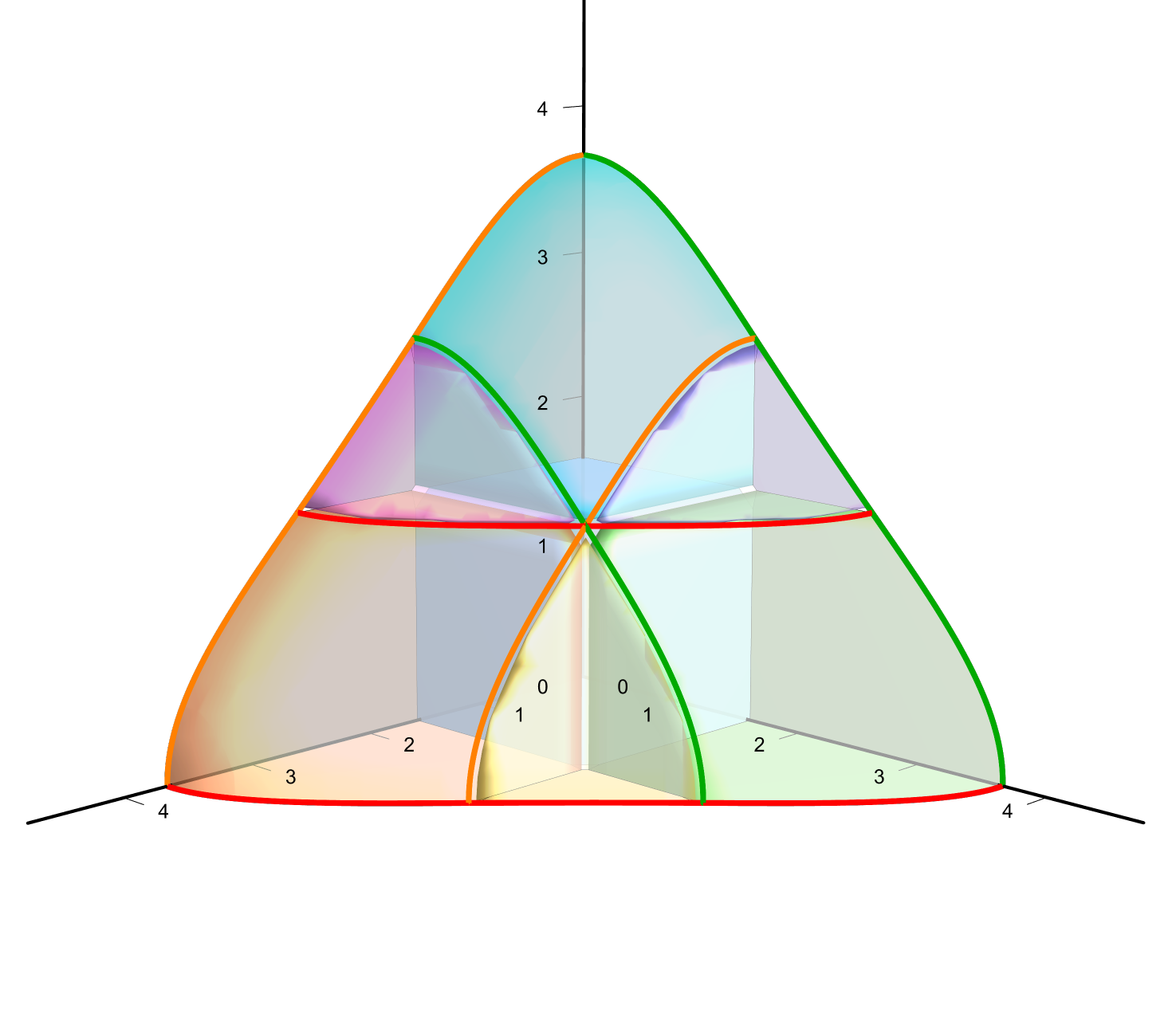}
\caption{\(\Omega^{\omega,+}_{\gamma}\), \(\gamma=1\).}
\label{fig:first_octant_gamma1}
\end{subfigure}
\hfill
\begin{subfigure}[b]{0.32\textwidth}
\centering
\includegraphics[width=\textwidth]{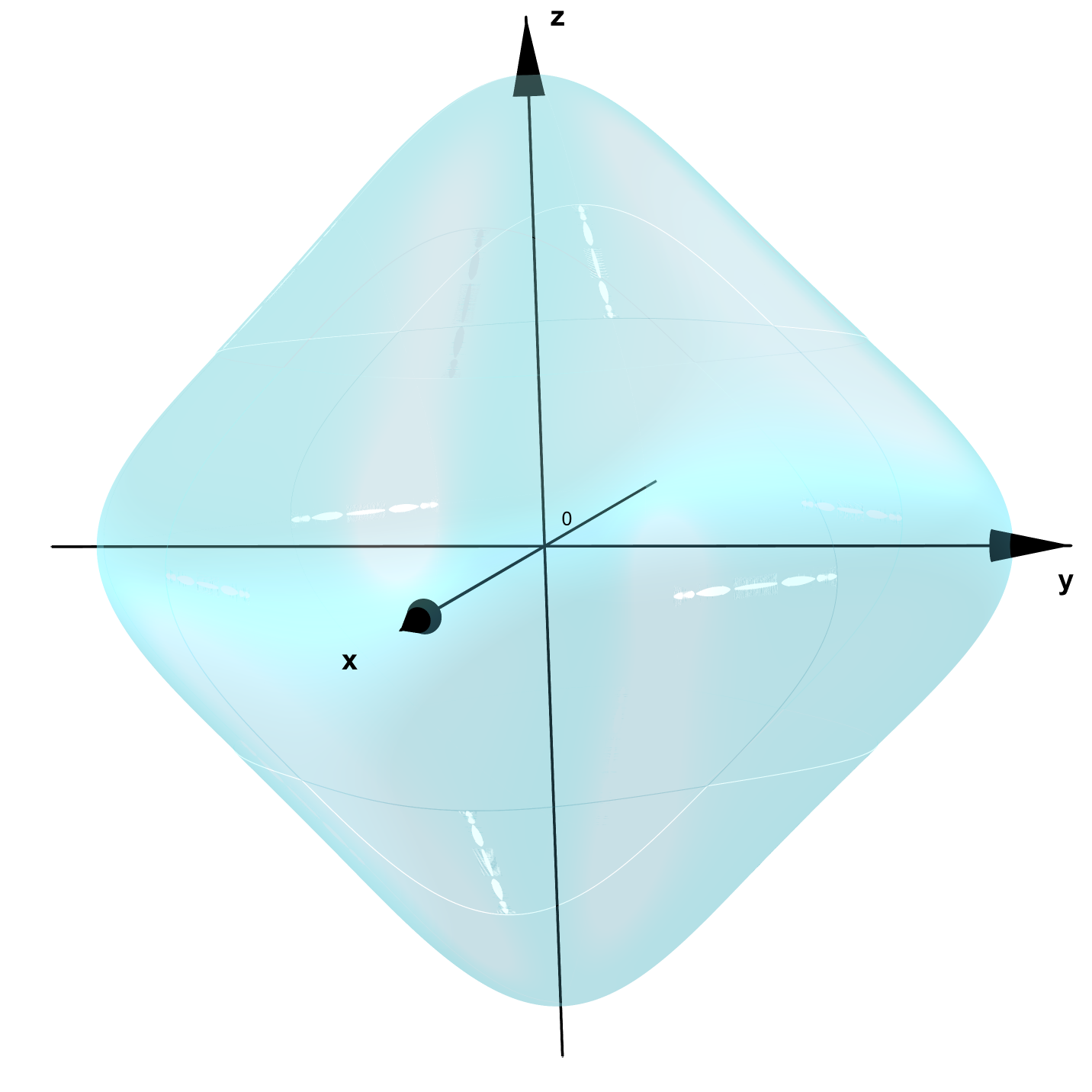}
\caption{\(\Omega^\omega_{\gamma}\), \(\gamma=1\).}
\label{fig:full_gamma1}
\end{subfigure}
\hfill
\begin{subfigure}[b]{0.32\textwidth}
\centering
\includegraphics[width=\textwidth]{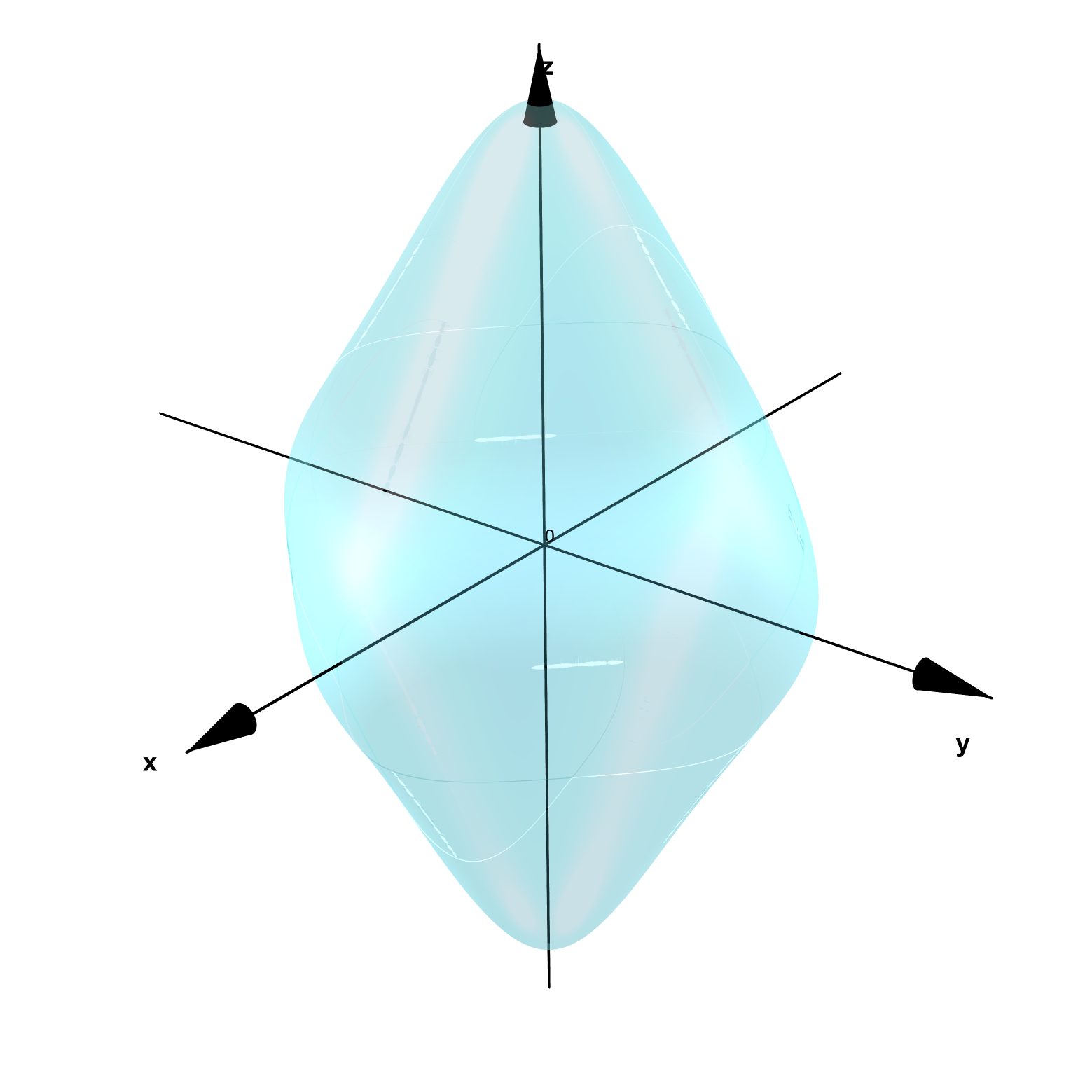}
\caption{\(\Omega^\omega_{\gamma}\), \(\gamma=1.5\).}
\label{fig:full_gamma15}
\end{subfigure}

\caption{Representative domains for \(\lambda=1\) and \(\Lambda=16\)
\((\omega=16)\). Left: the first-octant domain in the symmetric case
\(\gamma=1\). Middle: the full reflected domain in the symmetric case
\(\gamma=1\). Right: the full reflected domain in a nonsymmetric case,
\(\gamma=1.5\).}
\label{fig:three_domains}
\end{figure}

\begin{remark}[Comparison with the planar construction]
The present geometry is motivated by the two-dimensional construction of Birindelli--Leoni \cite{BirindelliLeoni2014}, but it no longer reduces to a one-profile deformation of a separable domain. In dimension two, the isotropic configuration can be rotated into a square, thereby recovering a Laplacian model with separated variables. In dimension three, no analogous simplification is available: the geometry retains its genuinely patchwise character, and the rigidity mechanism must come from the control of the Hessian inertia on the cube, the caps, and the bridges.

This is consistent with the discussion in the introduction. Although the family $\Omega^{\omega}_{\gamma}$ is built from a central cube, the surrounding attachments are intrinsically three-dimensional and are not designed to be reduced, by an orthogonal or affine change of variables, to a separable cubical model compatible with the Pucci operator.
\end{remark}

\section{Construction of the principal eigenfunction}
\label{sec:eigenfunctions}

In this section we construct explicitly a positive eigenfunction of
$-\mathcal M^+_{\lambda,\Lambda}$ in the reference domain $\Omega^\omega_\gamma$
introduced in Section~\ref{sec:construction}. Throughout this section, we fix
\[
\omega=\frac{\Lambda}{\lambda}\geq4,
\qquad
\gamma\in\Gamma_\omega,
\qquad
\alpha=\frac{\pi}{2}.
\]

We work first in the first octant. The function \(u^{\omega,+}_\gamma\),
defined patchwise on the seven regions
\[
\Omega_{\gamma}^{\omega,+}[C],\ \Omega_{\gamma}^{\omega,+}[X],\  \Omega_{\gamma}^{\omega,+}[Y] , \ \Omega_{\gamma}^{\omega,+}[Z], \ \Omega_{\gamma}^{\omega,+}[ZX],\ \Omega_{\gamma}^{\omega,+}[YZ], \ \Omega_{\gamma}^{\omega,+}[XY]
\]
will be shown to satisfy
\[
-\mathcal M^+_{\lambda,\Lambda}(D^2u^{\omega,+}_\gamma)=\lambda u^{\omega,+}_\gamma
\]
in each patch, to be positive in the interior, to vanish on the outer boundary,
and to glue at the \(C^1\) level across the common interfaces.

\[
u^{\omega,+}_\gamma(x,y,z)
=
\begin{cases}
\gamma\cos x+\gamma\cos y+\cos z,
& (x,y,z)\in\Omega_{\gamma}^{\omega,+}[C],\\[6pt]
\gamma\cos x+\gamma\cos y
+\sqrt{\omega}\cos\!\left(\dfrac{z-\alpha}{\sqrt{\omega}}+\dfrac{\pi}{2}\right),
& (x,y,z)\in\Omega_{\gamma}^{\omega,+}[Z],\\[8pt]
\gamma\sqrt{\omega}\cos\!\left(\dfrac{x-\alpha}{\sqrt{\omega}}+\dfrac{\pi}{2}\right)
+\gamma\cos y+\cos z,
& (x,y,z)\in \Omega_{\gamma}^{\omega,+}[X],\\[8pt]
\gamma\cos x
+\gamma\sqrt{\omega}\cos\!\left(\dfrac{y-\alpha}{\sqrt{\omega}}+\dfrac{\pi}{2}\right)
+\cos z,
& (x,y,z)\in\Omega_{\gamma}^{\omega,+}[Y],\\[8pt]
\gamma\sqrt{\omega}\cos\!\left(\dfrac{x-\alpha}{\sqrt{\omega}}+\dfrac{\pi}{2}\right)
+\gamma\cos y
+\sqrt{\omega}\cos\!\left(\dfrac{z-\alpha}{\sqrt{\omega}}+\dfrac{\pi}{2}\right),
& (x,y,z)\in\Omega_{\gamma}^{\omega,+}[ZX],\\[8pt]
\gamma\sqrt{\omega}\cos\!\left(\dfrac{x-\alpha}{\sqrt{\omega}}+\dfrac{\pi}{2}\right)
+\gamma\sqrt{\omega}\cos\!\left(\dfrac{y-\alpha}{\sqrt{\omega}}+\dfrac{\pi}{2}\right)
+\cos z,
& (x,y,z)\in\Omega_{\gamma}^{\omega,+}[XY],\\[8pt]
\gamma\cos x
+\gamma\sqrt{\omega}\cos\!\left(\dfrac{y-\alpha}{\sqrt{\omega}}+\dfrac{\pi}{2}\right)
+\sqrt{\omega}\cos\!\left(\dfrac{z-\alpha}{\sqrt{\omega}}+\dfrac{\pi}{2}\right),
& (x,y,z)\in\Omega_{\gamma}^{\omega,+}[YZ].
\end{cases}
\]
We shall show that $u^{\omega,+}_\gamma$ satisfies
\[
-\mathcal M^+_{\lambda,\Lambda}(D^2u^{\omega,+}_\gamma)=\lambda u^{\omega,+}_\gamma
\]
in each patch, is positive in the interior, vanishes on the outer boundary,
and matches at the $C^1$ level across the common interfaces.

\subsection{Patchwise verification of the equation}
We prove these properties in the next two subsections. For convenience, we denote by
\[
u_C,\quad u_Z,\quad u_X,\quad u_Y,\quad u_{ZX},\quad u_{XY},\quad u_{YZ}
\]
the restrictions of \(u^{\omega,+}_\gamma\) to
\[
\Omega_{\gamma}^{\omega,+}[C],\quad
\Omega_{\gamma}^{\omega,+}[Z],\quad
\Omega_{\gamma}^{\omega,+}[X],\quad
\Omega_{\gamma}^{\omega,+}[Y],\quad
\Omega_{\gamma}^{\omega,+}[ZX],\quad
\Omega_{\gamma}^{\omega,+}[XY],\quad
\Omega_{\gamma}^{\omega,+}[YZ],
\]
respectively.

\begin{proposition}[Central cube]
In \(\Omega_{\gamma}^{\omega,+}[C]\) one has
\[
-\mathcal M^+_{\lambda,\Lambda}(D^2u_C)=\lambda u_C.
\]
\end{proposition}

\begin{proof}
On \(\Omega_{\gamma}^{\omega,+}[C]\),
\[
u_C(x,y,z)=\gamma\cos x+\gamma\cos y+\cos z.
\]
Hence
\[
D^2u_C=\operatorname{diag}(-\gamma\cos x,\,-\gamma\cos y,\,-\cos z).
\]
Since \(0\le x,y,z\le \alpha=\pi/2\), all three eigenvalues of \(D^2u_C\) are nonpositive. Therefore
\[
-\mathcal M^+_{\lambda,\Lambda}(D^2u_C)
=
-\lambda\Delta u_C
=
\lambda u_C.
\]
\end{proof}

\begin{proposition}[The \(Z\)--cap]
In \(\Omega_{\gamma}^{\omega,+}[Z]\) one has
\[
-\mathcal M^+_{\lambda,\Lambda}(D^2u_Z)=\lambda u_Z.
\]
Moreover, \(u_Z>0\) in \(\Omega_{\gamma}^{\omega,+}[Z]^\circ\) and \(u_Z=0\) on the outer boundary portion of \(\Omega_{\gamma}^{\omega,+}[Z]\).
\end{proposition}

\begin{proof}
On \(\Omega_{\gamma}^{\omega,+}[Z]\),
\[
u_Z(x,y,z)=\gamma(\cos x+\cos y)-\sqrt{\omega}\sin\!\left(\frac{z-\alpha}{\sqrt{\omega}}\right).
\]
We compute
\[
(u_Z)_{xx}=-\gamma\cos x\le 0,\qquad
(u_Z)_{yy}=-\gamma\cos y\le 0,\qquad
(u_Z)_{zz}=\frac{1}{\sqrt{\omega}}\sin\!\left(\frac{z-\alpha}{\sqrt{\omega}}\right)\ge 0.
\]
Thus the Hessian has exactly one nonnegative eigenvalue, namely in the \(z\)-direction. Therefore
\[
-\mathcal M^+_{\lambda,\Lambda}(D^2u_Z)
=
-\Lambda (u_Z)_{zz}
-\lambda\big((u_Z)_{xx}+(u_Z)_{yy}\big).
\]
Using \(\Lambda=\omega\lambda\), we obtain
\[
-\Lambda (u_Z)_{zz}
=
-\lambda\sqrt{\omega}\sin\!\left(\frac{z-\alpha}{\sqrt{\omega}}\right),
\]
and
\[
-\lambda\big((u_Z)_{xx}+(u_Z)_{yy}\big)
=
\lambda\gamma(\cos x+\cos y).
\]
Hence
\[
-\mathcal M^+_{\lambda,\Lambda}(D^2u_Z)
=
\lambda\left[
\gamma(\cos x+\cos y)-\sqrt{\omega}\sin\!\left(\frac{z-\alpha}{\sqrt{\omega}}\right)
\right]
=
\lambda u_Z.
\]

on the outer boundary portion of \(\Omega_{\gamma}^{\omega,+}[Z]\),
\[
\sin\!\left(\frac{z-\alpha}{\sqrt{\omega}}\right)
=
\frac{\gamma}{\sqrt{\omega}}(\cos x+\cos y),
\]
and therefore \(u_Z=0\).

If \((x,y,z)\in \Omega_{\gamma}^{\omega,+}[Z]^\circ\), then
\[
0\le \frac{z-\alpha}{\sqrt{\omega}}
<
\arcsin\!\left(\frac{\gamma}{\sqrt{\omega}}(\cos x+\cos y)\right)
\le \frac{\pi}{2}.
\]
Since \(\sin\) is increasing on \([0,\pi/2]\), it follows that
\[
\sin\!\left(\frac{z-\alpha}{\sqrt{\omega}}\right)
<
\frac{\gamma}{\sqrt{\omega}}(\cos x+\cos y),
\]
hence \(u_Z>0\) in the interior.
\end{proof}

\begin{proposition}[The \(X\)--cap]
In \(\Omega_{\gamma}^{\omega,+}[X]\) one has
\[
-\mathcal M^+_{\lambda,\Lambda}(D^2u_X)=\lambda u_X.
\]
Moreover, \(u_X>0\) in \(\Omega_{\gamma}^{\omega,+}[X]^\circ\) and \(u_X=0\) on the outer boundary portion of \(\Omega_{\gamma}^{\omega,+}[X]\).
\end{proposition}

\begin{proof}
On \(\Omega_{\gamma}^{\omega,+}[X]\),
\[
u_X(x,y,z)
=
-\gamma\sqrt{\omega}\sin\!\left(\frac{x-\alpha}{\sqrt{\omega}}\right)
+\gamma\cos y+\cos z.
\]
We compute
\[
(u_X)_{xx}=\frac{\gamma}{\sqrt{\omega}}\sin\!\left(\frac{x-\alpha}{\sqrt{\omega}}\right)\ge 0,\qquad
(u_X)_{yy}=-\gamma\cos y\le 0,\qquad
(u_X)_{zz}=-\cos z\le 0.
\]
Thus the Hessian has exactly one nonnegative eigenvalue, namely in the \(x\)-direction. Hence
\[
-\mathcal M^+_{\lambda,\Lambda}(D^2u_X)
=
-\Lambda (u_X)_{xx}
-\lambda\big((u_X)_{yy}+(u_X)_{zz}\big).
\]
Using \(\Lambda=\omega\lambda\), we obtain
\[
-\Lambda (u_X)_{xx}
=
-\lambda\gamma\sqrt{\omega}\sin\!\left(\frac{x-\alpha}{\sqrt{\omega}}\right),
\]
and
\[
-\lambda\big((u_X)_{yy}+(u_X)_{zz}\big)
=
\lambda(\gamma\cos y+\cos z).
\]
Therefore
\[
-\mathcal M^+_{\lambda,\Lambda}(D^2u_X)
=
\lambda\left[
-\gamma\sqrt{\omega}\sin\!\left(\frac{x-\alpha}{\sqrt{\omega}}\right)+\gamma\cos y+\cos z
\right]
=
\lambda u_X.
\]

on the outer boundary portion of \(\Omega_{\gamma}^{\omega,+}[X]\),
\[
\sin\!\left(\frac{x-\alpha}{\sqrt{\omega}}\right)
=
\frac{\gamma\cos y+\cos z}{\gamma\sqrt{\omega}},
\]
and hence \(u_X=0\).

If \((x,y,z)\in \Omega_{\gamma}^{\omega,+}[X]^\circ\), then
\[
0\le \frac{x-\alpha}{\sqrt{\omega}}
<
\arcsin\!\left(\frac{\gamma\cos y+\cos z}{\gamma\sqrt{\omega}}\right)
\le \frac{\pi}{2},
\]
so
\[
\sin\!\left(\frac{x-\alpha}{\sqrt{\omega}}\right)
<
\frac{\gamma\cos y+\cos z}{\gamma\sqrt{\omega}}.
\]
Therefore \(u_X>0\) in the interior.
\end{proof}

\begin{proposition}[The \(Y\)--cap]
In \(\Omega_{\gamma}^{\omega,+}[Y]\) one has
\[
-\mathcal M^+_{\lambda,\Lambda}(D^2u_Y)=\lambda u_Y.
\]
Moreover, \(u_Y>0\) in \(\Omega_{\gamma}^{\omega,+}[Y]^\circ\) and \(u_Y=0\) on the outer boundary portion of \(\Omega_{\gamma}^{\omega,+}[Y]\).
\end{proposition}

\begin{proof}
On \(\Omega_{\gamma}^{\omega,+}[Y]\),
\[
u_Y(x,y,z)
=
\gamma\cos x
-\gamma\sqrt{\omega}\sin\!\left(\frac{y-\alpha}{\sqrt{\omega}}\right)
+\cos z.
\]
We compute
\[
(u_Y)_{xx}=-\gamma\cos x\le 0,\qquad
(u_Y)_{yy}=\frac{\gamma}{\sqrt{\omega}}\sin\!\left(\frac{y-\alpha}{\sqrt{\omega}}\right)\ge 0,\qquad
(u_Y)_{zz}=-\cos z\le 0.
\]
Thus the Hessian has exactly one nonnegative eigenvalue, namely in the \(y\)-direction. Therefore
\[
-\mathcal M^+_{\lambda,\Lambda}(D^2u_Y)
=
-\Lambda (u_Y)_{yy}
-\lambda\big((u_Y)_{xx}+(u_Y)_{zz}\big).
\]
Using \(\Lambda=\omega\lambda\), we obtain
\[
-\Lambda (u_Y)_{yy}
=
-\lambda\gamma\sqrt{\omega}\sin\!\left(\frac{y-\alpha}{\sqrt{\omega}}\right),
\]
and
\[
-\lambda\big((u_Y)_{xx}+(u_Y)_{zz}\big)
=
\lambda(\gamma\cos x+\cos z).
\]
Hence
\[
-\mathcal M^+_{\lambda,\Lambda}(D^2u_Y)
=
\lambda\left[
\gamma\cos x-\gamma\sqrt{\omega}\sin\!\left(\frac{y-\alpha}{\sqrt{\omega}}\right)+\cos z
\right]
=
\lambda u_Y.
\]

on the outer boundary portion of \(\Omega_{\gamma}^{\omega,+}[Y]\),
\[
\sin\!\left(\frac{y-\alpha}{\sqrt{\omega}}\right)
=
\frac{\gamma\cos x+\cos z}{\gamma\sqrt{\omega}},
\]
and hence \(u_Y=0\).

If \((x,y,z)\in \Omega_{\gamma}^{\omega,+}[Y]^\circ\), then
\[
0\le \frac{y-\alpha}{\sqrt{\omega}}
<
\arcsin\!\left(\frac{\gamma\cos x+\cos z}{\gamma\sqrt{\omega}}\right)
\le \frac{\pi}{2},
\]
so
\[
\sin\!\left(\frac{y-\alpha}{\sqrt{\omega}}\right)
<
\frac{\gamma\cos x+\cos z}{\gamma\sqrt{\omega}}.
\]
Therefore \(u_Y>0\) in the interior.
\end{proof}

\begin{proposition}[The \(ZX\)--bridge]
In \(\Omega_{\gamma}^{\omega,+}[ZX]\) one has
\[
-\mathcal M^+_{\lambda,\Lambda}(D^2u_{ZX})=\lambda u_{ZX}.
\]
Moreover, \(u_{ZX}>0\) in \(\Omega_{\gamma}^{\omega,+}[ZX]^\circ\) and \(u_{ZX}=0\) on the outer boundary portion of \(\Omega_{\gamma}^{\omega,+}[ZX]\).
\end{proposition}

\begin{proof}
On \(\Omega_{\gamma}^{\omega,+}[ZX]\), write
\[
u_{ZX}(x,y,z)
=
-\gamma\sqrt{\omega}\sin s+\gamma\cos y-\sqrt{\omega}\sin t,
\qquad
s:=\frac{x-\alpha}{\sqrt{\omega}},\quad
t:=\frac{z-\alpha}{\sqrt{\omega}}.
\]
By the definition of \(\Omega_{\gamma}^{\omega,+}[ZX]\),
\[
0\le s\le \arcsin\!\Big(\frac{1}{\sqrt{\omega}}\Big),
\qquad
0\le t\le \arcsin\!\Big(\frac{\gamma}{\sqrt{\omega}}\Big),
\]
hence \(\sin s,\sin t\ge 0\). We compute
\[
(u_{ZX})_{xx}=\frac{\gamma}{\sqrt{\omega}}\sin s\ge 0,\qquad
(u_{ZX})_{yy}=-\gamma\cos y\le 0,\qquad
(u_{ZX})_{zz}=\frac{1}{\sqrt{\omega}}\sin t\ge 0.
\]
Thus the Hessian has two nonnegative eigenvalues, in the \(x\)- and \(z\)-directions,
and one nonpositive eigenvalue, in the \(y\)-direction. Therefore
\[
-\mathcal M^+_{\lambda,\Lambda}(D^2u_{ZX})
=
-\Lambda\big((u_{ZX})_{xx}+(u_{ZX})_{zz}\big)-\lambda (u_{ZX})_{yy}.
\]
Using \(\Lambda=\omega\lambda\), we obtain
\[
-\Lambda (u_{ZX})_{xx}=-\lambda\gamma\sqrt{\omega}\sin s,\qquad
-\Lambda (u_{ZX})_{zz}=-\lambda\sqrt{\omega}\sin t,
\]
and
\[
-\lambda (u_{ZX})_{yy}=\lambda\gamma\cos y.
\]
Hence
\[
-\mathcal M^+_{\lambda,\Lambda}(D^2u_{ZX})
=
\lambda\big(-\gamma\sqrt{\omega}\sin s+\gamma\cos y-\sqrt{\omega}\sin t\big)
=
\lambda u_{ZX}.
\]

on the outer boundary portion of \(\Omega_{\gamma}^{\omega,+}[ZX]\),
\[
\cos y=\sqrt{\omega}\sin s+\frac{\sqrt{\omega}}{\gamma}\sin t,
\]
so that
\[
\gamma\cos y=\gamma\sqrt{\omega}\sin s+\sqrt{\omega}\sin t,
\]
and therefore \(u_{ZX}=0\).

If \((x,y,z)\in \Omega_{\gamma}^{\omega,+}[ZX]^\circ\), then
\[
y<\arccos\!\left(\sqrt{\omega}\sin s+\frac{\sqrt{\omega}}{\gamma}\sin t\right),
\]
hence
\[
\cos y>
\sqrt{\omega}\sin s+\frac{\sqrt{\omega}}{\gamma}\sin t.
\]
Multiplying by \(\gamma\), we obtain
\[
\gamma\cos y>\gamma\sqrt{\omega}\sin s+\sqrt{\omega}\sin t,
\]
and therefore \(u_{ZX}>0\) in the interior.
\end{proof}

\begin{proposition}[The \(XY\)--bridge]
In \(\Omega_{\gamma}^{\omega,+}[XY]\) one has
\[
-\mathcal M^+_{\lambda,\Lambda}(D^2u_{XY})=\lambda u_{XY}.
\]
Moreover, \(u_{XY}>0\) in \(\Omega_{\gamma}^{\omega,+}[XY]^\circ\) and \(u_{XY}=0\) on the outer boundary portion of \(\Omega_{\gamma}^{\omega,+}[XY]\).
\end{proposition}

\begin{proof}
On \(\Omega_{\gamma}^{\omega,+}[XY]\), write
\[
u_{XY}(x,y,z)
=
-\gamma\sqrt{\omega}\sin s-\gamma\sqrt{\omega}\sin r+\cos z,
\qquad
s:=\frac{x-\alpha}{\sqrt{\omega}},\quad
r:=\frac{y-\alpha}{\sqrt{\omega}}.
\]
By the definition of \(\Omega_{\gamma}^{\omega,+}[XY]\),
\[
0\le s,r\le \arcsin\!\Big(\frac{1}{\gamma\sqrt{\omega}}\Big),
\]
hence \(\sin s,\sin r\ge 0\). We compute
\[
(u_{XY})_{xx}=\frac{\gamma}{\sqrt{\omega}}\sin s\ge 0,\qquad
(u_{XY})_{yy}=\frac{\gamma}{\sqrt{\omega}}\sin r\ge 0,\qquad
(u_{XY})_{zz}=-\cos z\le 0.
\]
Thus the Hessian has two nonnegative eigenvalues, in the \(x\)- and \(y\)-directions,
and one nonpositive eigenvalue, in the \(z\)-direction. Therefore
\[
-\mathcal M^+_{\lambda,\Lambda}(D^2u_{XY})
=
-\Lambda\big((u_{XY})_{xx}+(u_{XY})_{yy}\big)-\lambda (u_{XY})_{zz}.
\]
Using \(\Lambda=\omega\lambda\), we obtain
\[
-\Lambda (u_{XY})_{xx}=-\lambda\gamma\sqrt{\omega}\sin s,\qquad
-\Lambda (u_{XY})_{yy}=-\lambda\gamma\sqrt{\omega}\sin r,
\]
and
\[
-\lambda (u_{XY})_{zz}=\lambda\cos z.
\]
Hence
\[
-\mathcal M^+_{\lambda,\Lambda}(D^2u_{XY})
=
\lambda\big(-\gamma\sqrt{\omega}\sin s-\gamma\sqrt{\omega}\sin r+\cos z\big)
=
\lambda u_{XY}.
\]

on the outer boundary portion of \(\Omega_{\gamma}^{\omega,+}[XY]\),
\[
\cos z=\gamma\sqrt{\omega}\sin s+\gamma\sqrt{\omega}\sin r,
\]
and therefore \(u_{XY}=0\).

If \((x,y,z)\in \Omega_{\gamma}^{\omega,+}[XY]^\circ\), then
\[
z<\arccos\!\left(\gamma\sqrt{\omega}\sin s+\gamma\sqrt{\omega}\sin r\right),
\]
so
\[
\cos z>\gamma\sqrt{\omega}\sin s+\gamma\sqrt{\omega}\sin r.
\]
Therefore \(u_{XY}>0\) in the interior.
\end{proof}

\begin{proposition}[The \(YZ\)--bridge]
In \(\Omega_{\gamma}^{\omega,+}[YZ]\) one has
\[
-\mathcal M^+_{\lambda,\Lambda}(D^2u_{YZ})=\lambda u_{YZ}.
\]
Moreover, \(u_{YZ}>0\) in \(\Omega_{\gamma}^{\omega,+}[YZ]^\circ\) and \(u_{YZ}=0\) on the outer boundary portion of \(\Omega_{\gamma}^{\omega,+}[YZ]\).
\end{proposition}

\begin{proof}
On \(\Omega_{\gamma}^{\omega,+}[YZ]\), write
\[
u_{YZ}(x,y,z)
=
\gamma\cos x-\gamma\sqrt{\omega}\sin r-\sqrt{\omega}\sin t,
\qquad
r:=\frac{y-\alpha}{\sqrt{\omega}},\quad
t:=\frac{z-\alpha}{\sqrt{\omega}}.
\]
By the definition of \(\Omega_{\gamma}^{\omega,+}[YZ]\),
\[
0\le r\le \arcsin\!\Big(\frac{1}{\sqrt{\omega}}\Big),
\qquad
0\le t\le \arcsin\!\Big(\frac{\gamma}{\sqrt{\omega}}\Big),
\]
hence \(\sin r,\sin t\ge 0\). We compute
\[
(u_{YZ})_{xx}=-\gamma\cos x\le 0,\qquad
(u_{YZ})_{yy}=\frac{\gamma}{\sqrt{\omega}}\sin r\ge 0,\qquad
(u_{YZ})_{zz}=\frac{1}{\sqrt{\omega}}\sin t\ge 0.
\]
Thus the Hessian has two nonnegative eigenvalues, in the \(y\)- and \(z\)-directions,
and one nonpositive eigenvalue, in the \(x\)-direction. Therefore
\[
-\mathcal M^+_{\lambda,\Lambda}(D^2u_{YZ})
=
-\Lambda\big((u_{YZ})_{yy}+(u_{YZ})_{zz}\big)-\lambda (u_{YZ})_{xx}.
\]
Using \(\Lambda=\omega\lambda\), we obtain
\[
-\Lambda (u_{YZ})_{yy}=-\lambda\gamma\sqrt{\omega}\sin r,\qquad
-\Lambda (u_{YZ})_{zz}=-\lambda\sqrt{\omega}\sin t,
\]
and
\[
-\lambda (u_{YZ})_{xx}=\lambda\gamma\cos x.
\]
Hence
\[
-\mathcal M^+_{\lambda,\Lambda}(D^2u_{YZ})
=
\lambda\big(\gamma\cos x-\gamma\sqrt{\omega}\sin r-\sqrt{\omega}\sin t\big)
=
\lambda u_{YZ}.
\]

on the outer boundary portion of \(\Omega_{\gamma}^{\omega,+}[YZ]\),
\[
\cos x=\sqrt{\omega}\sin r+\frac{\sqrt{\omega}}{\gamma}\sin t,
\]
so that
\[
\gamma\cos x=\gamma\sqrt{\omega}\sin r+\sqrt{\omega}\sin t,
\]
and therefore \(u_{YZ}=0\).

If \((x,y,z)\in \Omega_{\gamma}^{\omega,+}[YZ]^\circ\), then
\[
x<\arccos\!\left(\sqrt{\omega}\sin r+\frac{\sqrt{\omega}}{\gamma}\sin t\right),
\]
hence
\[
\cos x>
\sqrt{\omega}\sin r+\frac{\sqrt{\omega}}{\gamma}\sin t.
\]
Multiplying by \(\gamma\), we obtain
\[
\gamma\cos x>\gamma\sqrt{\omega}\sin r+\sqrt{\omega}\sin t,
\]
and therefore \(u_{YZ}>0\) in the interior.
\end{proof}

\subsection{\texorpdfstring{$C^1$}{C1} matching across interfaces}

We next verify that the patchwise definition of \(u^{\omega,+}_\gamma\) is globally of class
\(C^1\) in the first octant.

\begin{proposition}[\texorpdfstring{$C^1$}{C1} compatibility]
The local restrictions
\[
u_C,\quad u_Z,\quad u_X,\quad u_Y,\quad u_{ZX},\quad u_{XY},\quad u_{YZ}
\]
agree together with their first derivatives on all common interfaces of the seven--patch
decomposition of \(\Omega_{\gamma}^{\omega,+}\). Consequently,
\[
u^{\omega,+}_\gamma\in C^1(\overline{\Omega_{\gamma}^{\omega,+}}).
\]
\end{proposition}

\begin{proof}
We verify the matching patch by patch.

\smallskip
\noindent
\emph{Cube--cap interfaces.}

On \(z=\alpha\), one has
\[
u_Z(x,y,\alpha)=\gamma\cos x+\gamma\cos y=u_C(x,y,\alpha).
\]
Moreover,
\[
\partial_xu_Z=-\gamma\sin x=\partial_xu_C,\qquad
\partial_yu_Z=-\gamma\sin y=\partial_yu_C,
\]
and
\[
\partial_zu_Z(x,y,\alpha)
=
-\cos\!\left(\frac{z-\alpha}{\sqrt{\omega}}\right)\Big|_{z=\alpha}
=
-1
=
\partial_zu_C(x,y,\alpha).
\]

On \(x=\alpha\), one has
\[
u_X(\alpha,y,z)=\gamma\cos y+\cos z=u_C(\alpha,y,z).
\]
Furthermore,
\[
\partial_xu_X(\alpha,y,z)
=
-\gamma\cos\!\left(\frac{x-\alpha}{\sqrt{\omega}}\right)\Big|_{x=\alpha}
=
-\gamma
=
\partial_xu_C(\alpha,y,z),
\]
while
\[
\partial_yu_X=-\gamma\sin y=\partial_yu_C,\qquad
\partial_zu_X=-\sin z=\partial_zu_C.
\]

On \(y=\alpha\), one has
\[
u_Y(x,\alpha,z)=\gamma\cos x+\cos z=u_C(x,\alpha,z).
\]
Moreover,
\[
\partial_xu_Y=-\gamma\sin x=\partial_xu_C,\qquad
\partial_zu_Y=-\sin z=\partial_zu_C,
\]
and
\[
\partial_yu_Y(x,\alpha,z)
=
-\gamma\cos\!\left(\frac{y-\alpha}{\sqrt{\omega}}\right)\Big|_{y=\alpha}
=
-\gamma
=
\partial_yu_C(x,\alpha,z).
\]

\smallskip
\noindent
\emph{Bridge--cap interfaces.}

On \(z=\alpha\), the \(ZX\)-bridge meets the \(X\)-cap. Indeed,
\[
u_{ZX}(x,y,\alpha)
=
-\gamma\sqrt{\omega}\sin\!\left(\frac{x-\alpha}{\sqrt{\omega}}\right)+\gamma\cos y
=
u_X(x,y,\alpha).
\]
Also,
\[
\partial_xu_{ZX}(x,y,\alpha)
=
-\gamma\cos\!\left(\frac{x-\alpha}{\sqrt{\omega}}\right)
=
\partial_xu_X(x,y,\alpha),
\]
\[
\partial_yu_{ZX}(x,y,\alpha)=-\gamma\sin y=\partial_yu_X(x,y,\alpha),
\]
and
\[
\partial_zu_{ZX}(x,y,\alpha)
=
-\cos\!\left(\frac{z-\alpha}{\sqrt{\omega}}\right)\Big|_{z=\alpha}
=
-1
=
\partial_zu_X(x,y,\alpha).
\]

On \(x=\alpha\), the \(ZX\)-bridge meets the \(Z\)-cap. Indeed,
\[
u_{ZX}(\alpha,y,z)
=
\gamma\cos y-\sqrt{\omega}\sin\!\left(\frac{z-\alpha}{\sqrt{\omega}}\right)
=
u_Z(\alpha,y,z).
\]
Moreover,
\[
\partial_xu_{ZX}(\alpha,y,z)
=
-\gamma\cos\!\left(\frac{x-\alpha}{\sqrt{\omega}}\right)\Big|_{x=\alpha}
=
-\gamma
=
\partial_xu_Z(\alpha,y,z),
\]
\[
\partial_yu_{ZX}(\alpha,y,z)=-\gamma\sin y=\partial_yu_Z(\alpha,y,z),
\]
and
\[
\partial_zu_{ZX}(\alpha,y,z)
=
-\cos\!\left(\frac{z-\alpha}{\sqrt{\omega}}\right)
=
\partial_zu_Z(\alpha,y,z).
\]

On \(y=\alpha\), the \(XY\)-bridge meets the \(X\)-cap. Indeed,
\[
u_{XY}(x,\alpha,z)
=
-\gamma\sqrt{\omega}\sin\!\left(\frac{x-\alpha}{\sqrt{\omega}}\right)+\cos z
=
u_X(x,\alpha,z).
\]
Also,
\[
\partial_xu_{XY}(x,\alpha,z)
=
-\gamma\cos\!\left(\frac{x-\alpha}{\sqrt{\omega}}\right)
=
\partial_xu_X(x,\alpha,z),
\]
\[
\partial_zu_{XY}(x,\alpha,z)=-\sin z=\partial_zu_X(x,\alpha,z),
\]
and
\[
\partial_yu_{XY}(x,\alpha,z)
=
-\gamma\cos\!\left(\frac{y-\alpha}{\sqrt{\omega}}\right)\Big|_{y=\alpha}
=
-\gamma
=
\partial_yu_X(x,\alpha,z).
\]

On \(x=\alpha\), the \(XY\)-bridge meets the \(Y\)-cap. Indeed,
\[
u_{XY}(\alpha,y,z)
=
-\gamma\sqrt{\omega}\sin\!\left(\frac{y-\alpha}{\sqrt{\omega}}\right)+\cos z
=
u_Y(\alpha,y,z).
\]
Moreover,
\[
\partial_yu_{XY}(\alpha,y,z)
=
-\gamma\cos\!\left(\frac{y-\alpha}{\sqrt{\omega}}\right)
=
\partial_yu_Y(\alpha,y,z),
\]
\[
\partial_zu_{XY}(\alpha,y,z)=-\sin z=\partial_zu_Y(\alpha,y,z),
\]
and
\[
\partial_xu_{XY}(\alpha,y,z)
=
-\gamma\cos\!\left(\frac{x-\alpha}{\sqrt{\omega}}\right)\Big|_{x=\alpha}
=
-\gamma
=
\partial_xu_Y(\alpha,y,z).
\]

On \(z=\alpha\), the \(YZ\)-bridge meets the \(Y\)-cap. Indeed,
\[
u_{YZ}(x,y,\alpha)
=
\gamma\cos x-\gamma\sqrt{\omega}\sin\!\left(\frac{y-\alpha}{\sqrt{\omega}}\right)
=
u_Y(x,y,\alpha).
\]
Also,
\[
\partial_xu_{YZ}(x,y,\alpha)=-\gamma\sin x=\partial_xu_Y(x,y,\alpha),
\]
\[
\partial_yu_{YZ}(x,y,\alpha)
=
-\gamma\cos\!\left(\frac{y-\alpha}{\sqrt{\omega}}\right)
=
\partial_yu_Y(x,y,\alpha),
\]
and
\[
\partial_zu_{YZ}(x,y,\alpha)
=
-\cos\!\left(\frac{z-\alpha}{\sqrt{\omega}}\right)\Big|_{z=\alpha}
=
-1
=
\partial_zu_Y(x,y,\alpha).
\]

On \(y=\alpha\), the \(YZ\)-bridge meets the \(Z\)-cap. Indeed,
\[
u_{YZ}(x,\alpha,z)
=
\gamma\cos x-\sqrt{\omega}\sin\!\left(\frac{z-\alpha}{\sqrt{\omega}}\right)
=
u_Z(x,\alpha,z).
\]
Moreover,
\[
\partial_xu_{YZ}(x,\alpha,z)=-\gamma\sin x=\partial_xu_Z(x,\alpha,z),
\]
\[
\partial_zu_{YZ}(x,\alpha,z)
=
-\cos\!\left(\frac{z-\alpha}{\sqrt{\omega}}\right)
=
\partial_zu_Z(x,\alpha,z),
\]
and
\[
\partial_yu_{YZ}(x,\alpha,z)
=
-\gamma\cos\!\left(\frac{y-\alpha}{\sqrt{\omega}}\right)\Big|_{y=\alpha}
=
-\gamma
=
\partial_yu_Z(x,\alpha,z).
\]

Hence the local pieces agree together with their first derivatives on every common interface. The same identities also give compatibility at the lower-dimensional
junctions where several interfaces meet. Therefore
\[
u^{\omega,+}_\gamma\in C^1(\overline{\Omega_{\gamma}^{\omega,+}}).
\]
\end{proof}

\subsection{Global extension and the principal eigenpair}

We now extend \(u^{\omega,+}_\gamma\) evenly across the three coordinate
planes and denote the resulting function on \(\Omega^\omega_\gamma\) by
\(u^\omega_\gamma\). Since the normal derivatives vanish on the coordinate
planes, the even reflection preserves the \(C^1\) regularity. Hence
\[
u^\omega_\gamma\in C^1(\Omega^\omega_\gamma).
\]
Moreover, \(u^\omega_\gamma\) is smooth in the interior of each reflected
patch, strictly positive in \(\Omega^\omega_\gamma\), and vanishes on
\(\partial\Omega^\omega_\gamma\).

By the patchwise verification established above, each local reflected piece
satisfies
\[
-M^+_{\lambda,\Lambda}(D^2u^\omega_\gamma)
=
\lambda u^\omega_\gamma
\]
classically on its own patch. Since the pieces glue at the \(C^1\) level
across the interfaces, Lemma~\ref{lem:gluing} yields that
\(u^\omega_\gamma\) is a global viscosity solution in \(\Omega^\omega_\gamma\).

\begin{theorem}\label{thm:explicit-eigenpair}
For every \(\gamma\in\Gamma_\omega\), the function \(u^\omega_\gamma\) is a positive eigenfunction of
\(-\mathcal M^+_{\lambda,\Lambda}\) in \(\Omega_{\gamma}^{\omega}\) associated with the eigenvalue
\(\lambda\). In particular,
\[
\mu_1^+(\Omega_{\gamma}^{\omega})=\lambda.
\]
\end{theorem}

\begin{proof}
By the preceding propositions, each local piece of \(u^\omega_\gamma\) satisfies the equation
with eigenvalue \(\lambda\), is positive in the interior of its patch, and vanishes on the outer
boundary. By the \(C^1\)-compatibility proposition, these pieces glue together across every
common interface. Hence \(u^\omega_\gamma\) is a positive viscosity solution of
\[
-\mathcal M^+_{\lambda,\Lambda}(D^2u^\omega_\gamma)=\lambda u^\omega_\gamma
\qquad\text{in }\Omega_{\gamma}^{\omega},
\]
with
\[
u^\omega_\gamma=0
\qquad\text{on }\partial\Omega_{\gamma}^{\omega}.
\]

Therefore \(\lambda\) is an eigenvalue of \(-\mathcal M^+_{\lambda,\Lambda}\) in
\(\Omega_{\gamma}^{\omega}\) associated with a positive eigenfunction. By the standard
characterization and simplicity of the principal positive half-eigenvalue, it follows that
\[
\mu_1^+(\Omega_{\gamma}^{\omega})=\lambda.
\]
\end{proof}

\section{Shear deformations and a spectral lower bound}

Throughout this section we keep the admissible range of the
three-dimensional construction:
\[
\omega=\frac{\Lambda}{\lambda}\ge 4,
\qquad
\gamma\in\Gamma_\omega
=
\left[
\frac{2}{\sqrt\omega},
\frac{\sqrt\omega}{2}
\right].
\]
For $|a|<\pi$ define
\[
s(a):=\sqrt{1-\Big(\frac{a}{\pi}\Big)^2}, 
\qquad 
\kappa(a):=\frac{\pi}{\sqrt{\pi^2-a^2}}=\frac{1}{s(a)}.
\]
Note that $0<s(a)\le1$ and $\kappa(a)\ge1$, with equality if and only if $a=0$.

\medskip

We consider the shear–scaling matrix
\[
C_a=
\begin{pmatrix}
s(a)&0&0\\[3pt]
\frac{a}{\pi}&1&0\\[3pt]
0&0&s(a)
\end{pmatrix},
\qquad
C_a^{-1}=
\begin{pmatrix}
\kappa(a)&0&0\\[3pt]
-\frac{a}{\sqrt{\pi^2-a^2}}&1&0\\[3pt]
0&0&\kappa(a)
\end{pmatrix}.
\]
A direct computation shows
\[
\det C_a = s(a)^2,
\]
so $C_a$ distorts volume unless $a=0$.

\medskip
For \(\gamma\in\Gamma_\omega\), we define the sheared domain
\[
\Omega^\omega_{\gamma,a}:=C_a(\Omega^\omega_\gamma),
\]
and the associated function
\[
u^\omega_{\gamma,a}(X):=
u^\omega_\gamma(C_a^{-1}X),
\qquad X\in\Omega^\omega_{\gamma,a}.
\]
Since \(u^\omega_\gamma>0\) in \(\Omega^\omega_\gamma\) and
\(u^\omega_\gamma=0\) on \(\partial\Omega^\omega_\gamma\), it follows that
\[
u^\omega_{\gamma,a}>0 \quad\text{in }\Omega^\omega_{\gamma,a},
\qquad
u^\omega_{\gamma,a}=0 \quad\text{on }\partial\Omega^\omega_{\gamma,a}.
\]


\begin{lemma}[Hessian under affine transformation]
Let $u\in C^2(\Omega)$ and let $A\in GL(3,\mathbb{R})$. 
Define $v(x)=u(A^{-1}x)$ on $A\Omega$. Then
\[
D^2 v(x)
= A^{-T}\big(D^2 u(A^{-1}x)\big)A^{-1}.
\]
\end{lemma}

Applying this with $A=C_a$, we obtain
\begin{equation}\label{eq:hessian-shear}
D^2 u_{\gamma,a}^{\omega}(X)
=
C_a^{-T}
\big(D^2 u_\gamma^\omega(C_a^{-1}X)\big)
C_a^{-1}.
\end{equation}

On each patch of \(\Omega^\omega_\gamma\), the matrix
\(D^2u^\omega_\gamma\) is diagonal in the \((x,y,z)\)-coordinates. Writing
all derivatives at the point
\[
(x,y,z)=C_a^{-1}X,
\]
set
\[
p:=(u^\omega_\gamma)_{xx},
\qquad
q:=(u^\omega_\gamma)_{yy},
\qquad
r:=(u^\omega_\gamma)_{zz}.
\]
Then
\[
D^2u^\omega_{\gamma,a}(X)
=
\begin{pmatrix}
\kappa(a)^2p+\dfrac{a^2}{\pi^2-a^2}q
&
-\dfrac{a}{\sqrt{\pi^2-a^2}}q
&
0
\\[1.1em]
-\dfrac{a}{\sqrt{\pi^2-a^2}}q
&
q
&
0
\\[0.8em]
0&0&\kappa(a)^2r
\end{pmatrix}.
\]
By Sylvester's law of inertia, the numbers of positive, negative, and zero
eigenvalues of \(D^2u^\omega_{\gamma,a}\) coincide with those of
\(D^2u^\omega_\gamma\).

Let \(B_a\) denote the upper-left \(2\times2\) block. Then
\[
\operatorname{tr}(B_a)=\kappa(a)^2(p+q),
\qquad
\det(B_a)=\kappa(a)^2pq.
\]
If
\[
\beta(a):=\frac{2a^2}{\pi^2}-1\in[-1,1),
\]
then the eigenvalues of \(B_a\) are
\[
\mu_\pm
=
\frac{\kappa(a)^2}{2}
\left[
p+q
\pm
\sqrt{p^2+q^2+2\beta(a)pq}
\right],
\]
while the third eigenvalue is
\[
\mu_3=\kappa(a)^2r.
\]


\begin{lemma}[Positive supersolution characterization of the principal eigenvalue]\label{lem:BNV-strict}
Let $\Omega \subset \mathbb{R}^n$ be a bounded domain, and let
\[
\psi \in C(\overline{\Omega}), \qquad \psi>0 \ \text{in } \Omega, \qquad \psi=0 \ \text{on } \partial\Omega,
\]
be such that
\[
-\mathcal{M}^{+}_{\lambda,\Lambda}(D^2\psi)\ge \mu\,\psi \qquad \text{in } \Omega
\]
in the viscosity sense, for some $\mu\in\mathbb{R}$. Then
\[
\mu \le \mu_1^{+}(\Omega).
\]
Moreover, if $\mu=\mu_1^{+}(\Omega)$, then $\psi$ is a principal eigenfunction corresponding to $\mu_1^{+}(\Omega)$.
\end{lemma}

\begin{proof}
This is the standard characterization of the principal positive half-eigenvalue; see, for example, Corollary~2.1 in \cite{BerestyckiNirenbergVaradhan1994} or Theorem~4.4 in \cite{QuaasSirakov2008}.
\end{proof}

\begin{theorem}[Spectral lower bound under shear]\label{thm:shear LB}
For every \(\gamma\in\Gamma_\omega\) and every \(|a|<\pi\), the function
\(u^\omega_{\gamma,a}\) satisfies
\[
-M^+_{\lambda,\Lambda}\big(D^2u^\omega_{\gamma,a}\big)
\ge
\frac{\lambda\pi^2}{\pi^2-a^2}\,u^\omega_{\gamma,a}
\qquad\text{in }\Omega^\omega_{\gamma,a}
\]
in the viscosity sense. Consequently,
\[
\mu_1^+(\Omega^\omega_{\gamma,a})
\ge
\frac{\lambda\pi^2}{\pi^2-a^2}.
\]
Moreover, in the admissible family \(\omega\ge4\), equality holds if and only
if \(a=0\).
\end{theorem}

\begin{proof}
We argue patchwise according to the sign pattern of
\[
p=(u^\omega_\gamma)_{xx},
\qquad
q=(u^\omega_\gamma)_{yy},
\qquad
r=(u^\omega_\gamma)_{zz}.
\]
By Section~4, only the following seven sign patterns occur:
\[
\begin{array}{c|c}
\text{Patch} & \text{Sign pattern of }(p,q,r)\\
\hline
\Omega^{\omega,+}_\gamma[C] & (-,-,-)\\
\Omega^{\omega,+}_\gamma[Z] & (-,-,+)\\
\Omega^{\omega,+}_\gamma[X] & (+,-,-)\\
\Omega^{\omega,+}_\gamma[Y] & (-,+,-)\\
\Omega^{\omega,+}_\gamma[XY] & (+,+,-)\\
\Omega^{\omega,+}_\gamma[ZX] & (+,-,+)\\
\Omega^{\omega,+}_\gamma[YZ] & (-,+,+).
\end{array}
\]

Recall that the eigenvalues of the upper-left \(2\times2\) block are
\[
\mu_\pm
=
\frac{\kappa(a)^2}{2}
\left[
p+q
\pm
\sqrt{p^2+q^2+2\beta(a)pq}
\right],
\]
and the third eigenvalue is
\[
\mu_3=\kappa(a)^2r.
\]

\medskip
\noindent
\textbf{Case 1: \(p\) and \(q\) have the same sign.}
In this case \(\mu_+\) and \(\mu_-\) have the same sign as \(p\) and \(q\),
because
\[
\mu_++\mu_-=\kappa(a)^2(p+q),
\qquad
\mu_+\mu_-=\kappa(a)^2pq.
\]

If \((p,q,r)=(-,-,-)\), that is, on the central cube, then all three
eigenvalues are nonpositive. Hence
\[
-M^+_{\lambda,\Lambda}(D^2u^\omega_{\gamma,a})
=
-\lambda(\mu_++\mu_-+\mu_3)
=
-\kappa(a)^2\lambda(p+q+r).
\]
Since on the cube
\[
-\lambda(p+q+r)=\lambda u^\omega_\gamma,
\]
we obtain
\[
-M^+_{\lambda,\Lambda}(D^2u^\omega_{\gamma,a})
=
\kappa(a)^2\lambda u^\omega_\gamma
=
\frac{\lambda\pi^2}{\pi^2-a^2}u^\omega_{\gamma,a}.
\]

If \((p,q,r)=(-,-,+)\), that is, on the \(Z\)-cap, then
\(\mu_\pm\le0\) and \(\mu_3\ge0\). Therefore
\[
-M^+_{\lambda,\Lambda}(D^2u^\omega_{\gamma,a})
=
-\lambda(\mu_++\mu_-)-\Lambda\mu_3
=
-\kappa(a)^2\big(\lambda(p+q)+\Lambda r\big).
\]
Since on the \(Z\)-cap
\[
-\lambda(p+q)-\Lambda r=\lambda u^\omega_\gamma,
\]
the same equality follows.

If \((p,q,r)=(+,+,-)\), that is, on the \(XY\)-bridge, then
\(\mu_\pm\ge0\) and \(\mu_3\le0\). Hence
\[
-M^+_{\lambda,\Lambda}(D^2u^\omega_{\gamma,a})
=
-\Lambda(\mu_++\mu_-)-\lambda\mu_3
=
-\kappa(a)^2\big(\Lambda(p+q)+\lambda r\big).
\]
Since on the \(XY\)-bridge
\[
-\Lambda(p+q)-\lambda r=\lambda u^\omega_\gamma,
\]
we again obtain
\[
-M^+_{\lambda,\Lambda}(D^2u^\omega_{\gamma,a})
=
\frac{\lambda\pi^2}{\pi^2-a^2}u^\omega_{\gamma,a}.
\]

\medskip
\noindent
\textbf{Case 2: \(p\) and \(q\) have opposite signs.}
This occurs on the \(X\)- and \(Y\)-caps and on the \(ZX\)- and
\(YZ\)-bridges. Set
\[
d_a:=\sqrt{p^2+q^2+2\beta(a)pq}.
\]
Since \(pq<0\) and \(\beta(a)\le1\), we have
\[
d_a\le |p-q|.
\]
Moreover, if \(a\ne0\), then \(\beta(a)>-1\), and the above inequality is
strict at every point where \(pq<0\).

First suppose \(r<0\), which corresponds to the \(X\)- and \(Y\)-caps. Then
\[
-M^+_{\lambda,\Lambda}(D^2u^\omega_{\gamma,a})
=
-\Lambda\mu_+-\lambda\mu_- -\lambda\mu_3.
\]
Using the formulas for \(\mu_\pm\) and \(\mu_3\), we get
\[
-M^+_{\lambda,\Lambda}(D^2u^\omega_{\gamma,a})
=
-\kappa(a)^2\lambda
\left[
\frac{\omega+1}{2}(p+q)
+
\frac{\omega-1}{2}d_a
+
r
\right].
\]

On the \(X\)-cap we have \(p>0>q\). Thus \(d_a\le p-q\), and hence
\[
-M^+_{\lambda,\Lambda}(D^2u^\omega_{\gamma,a})
\ge
-\kappa(a)^2\lambda(\omega p+q+r).
\]
Since on the \(X\)-cap
\[
-\Lambda p-\lambda q-\lambda r=\lambda u^\omega_\gamma,
\]
we obtain
\[
-M^+_{\lambda,\Lambda}(D^2u^\omega_{\gamma,a})
\ge
\frac{\lambda\pi^2}{\pi^2-a^2}u^\omega_{\gamma,a}.
\]

On the \(Y\)-cap we have \(q>0>p\). Thus \(d_a\le q-p\), and similarly
\[
-M^+_{\lambda,\Lambda}(D^2u^\omega_{\gamma,a})
\ge
-\kappa(a)^2\lambda(p+\omega q+r).
\]
Since on the \(Y\)-cap
\[
-\lambda p-\Lambda q-\lambda r=\lambda u^\omega_\gamma,
\]
the same lower bound follows.

Now suppose \(r>0\), which corresponds to the \(ZX\)- and \(YZ\)-bridges.
Then
\[
-M^+_{\lambda,\Lambda}(D^2u^\omega_{\gamma,a})
=
-\Lambda\mu_+-\lambda\mu_- -\Lambda\mu_3.
\]
Therefore
\[
-M^+_{\lambda,\Lambda}(D^2u^\omega_{\gamma,a})
=
-\kappa(a)^2\lambda
\left[
\frac{\omega+1}{2}(p+q)
+
\frac{\omega-1}{2}d_a
+
\omega r
\right].
\]

On the \(ZX\)-bridge we have \(p>0>q\). Hence \(d_a\le p-q\), and
\[
-M^+_{\lambda,\Lambda}(D^2u^\omega_{\gamma,a})
\ge
-\kappa(a)^2\lambda(\omega p+q+\omega r).
\]
Since on the \(ZX\)-bridge
\[
-\Lambda p-\lambda q-\Lambda r=\lambda u^\omega_\gamma,
\]
the desired lower bound follows.

On the \(YZ\)-bridge we have \(q>0>p\). Hence \(d_a\le q-p\), and
\[
-M^+_{\lambda,\Lambda}(D^2u^\omega_{\gamma,a})
\ge
-\kappa(a)^2\lambda(p+\omega q+\omega r).
\]
Since on the \(YZ\)-bridge
\[
-\lambda p-\Lambda q-\Lambda r=\lambda u^\omega_\gamma,
\]
we again obtain
\[
-M^+_{\lambda,\Lambda}(D^2u^\omega_{\gamma,a})
\ge
\frac{\lambda\pi^2}{\pi^2-a^2}u^\omega_{\gamma,a}.
\]

Combining the cases, we obtain the patchwise inequality
\[
-M^+_{\lambda,\Lambda}(D^2u^\omega_{\gamma,a})
\ge
\frac{\lambda\pi^2}{\pi^2-a^2}u^\omega_{\gamma,a}.
\]
Since the pieces agree at the \(C^1\) level across the interfaces, the same
test-function argument as in Lemma~\ref{lem:gluing} gives the inequality in
the viscosity sense throughout \(\Omega^\omega_{\gamma,a}\).

Since \(u^\omega_{\gamma,a}>0\) in \(\Omega^\omega_{\gamma,a}\), Lemma~5.2
yields
\[
\mu_1^+(\Omega^\omega_{\gamma,a})
\ge
\frac{\lambda\pi^2}{\pi^2-a^2}.
\]

It remains to characterize the equality case. If \(a=0\), then \(C_a\) is
the identity and
\[
u^\omega_{\gamma,0}=u^\omega_\gamma.
\]
By Theorem~4.9,
\[
\mu_1^+(\Omega^\omega_\gamma)=\lambda,
\]
so equality holds in the lower bound.

Conversely, assume \(a\ne0\). Since \(\omega\ge4\), in particular
\(\omega>1\). In the interior of the \(X\)-cap one has
\[
p>0,\qquad q<0,\qquad r<0,
\]
with \(pq<0\). Since \(a\ne0\), we have
\[
\sqrt{p^2+q^2+2\beta(a)pq}<p-q
\]
at such points. Hence the differential inequality above is strict at such
points:
\[
-M^+_{\lambda,\Lambda}(D^2u^\omega_{\gamma,a})
>
\frac{\lambda\pi^2}{\pi^2-a^2}u^\omega_{\gamma,a}.
\]
Thus \(u^\omega_{\gamma,a}\) cannot be an eigenfunction corresponding to
\(\lambda\pi^2/(\pi^2-a^2)\). By Lemma~5.2, equality in the lower bound for
\(\mu_1^+(\Omega^\omega_{\gamma,a})\) would force \(u^\omega_{\gamma,a}\) to
be such an eigenfunction. Therefore equality is impossible when \(a\ne0\).

Hence equality holds if and only if \(a=0\).
\end{proof}

\begin{corollary}[Optimality of the unsheared member for fixed \(\gamma\)]
Fix \(\omega\ge4\), \(\gamma\in\Gamma_\omega\), and a target volume \(V>0\).
Among the family
\[
\left\{
\delta\Omega^\omega_{\gamma,a}:
|a|<\pi,\ \delta>0,\ 
|\delta\Omega^\omega_{\gamma,a}|=V
\right\},
\]
the principal eigenvalue is uniquely minimized when \(a=0\).
\end{corollary}

\begin{proof}
It is enough to work with the volume-one normalization, since changing the
target volume only multiplies all eigenvalues by the same scaling factor.

Since
\[
\det C_a=s(a)^2=1-\frac{a^2}{\pi^2},
\]
we have
\[
|\Omega^\omega_{\gamma,a}|
=
\left(1-\frac{a^2}{\pi^2}\right)|\Omega^\omega_\gamma|.
\]
For the volume-normalized domain
\[
\widehat{\Omega}^\omega_{\gamma,a}
:=
\frac{\Omega^\omega_{\gamma,a}}
{|\Omega^\omega_{\gamma,a}|^{1/3}},
\]
the scaling law gives
\[
\mu_1^+(\widehat{\Omega}^\omega_{\gamma,a})
=
|\Omega^\omega_{\gamma,a}|^{2/3}
\mu_1^+(\Omega^\omega_{\gamma,a}).
\]
By Theorem~5.3,
\[
\mu_1^+(\widehat{\Omega}^\omega_{\gamma,a})
\ge
|\Omega^\omega_{\gamma,a}|^{2/3}
\frac{\lambda\pi^2}{\pi^2-a^2}.
\]
Using the volume formula, we obtain
\[
|\Omega^\omega_{\gamma,a}|^{2/3}
\frac{\lambda\pi^2}{\pi^2-a^2}
=
\lambda|\Omega^\omega_\gamma|^{2/3}
\left(1-\frac{a^2}{\pi^2}\right)^{-1/3}.
\]
The factor
\[
\left(1-\frac{a^2}{\pi^2}\right)^{-1/3}
\]
is uniquely minimized at \(a=0\), and it is strictly larger than \(1\) when
\(a\ne0\). Since equality is attained at \(a=0\) by Theorem~4.9, the
volume-normalized principal eigenvalue is uniquely minimized at \(a=0\).
\end{proof}

\begin{corollary}[Symmetric minimizer]\label{cor:min-at-1}
Let \(\omega=\Lambda/\lambda\ge4\). Among all volume-normalized domains
\[
\left\{
\frac{\Omega^\omega_{\gamma,a}}
{|\Omega^\omega_{\gamma,a}|^{1/3}}
:
\gamma\in\Gamma_\omega,\ |a|<\pi
\right\},
\]
the principal eigenvalue attains its unique minimum at the symmetric
unsheared configuration \((\gamma,a)=(1,0)\). Equivalently,
\[
\mu_1^+
\left(
\frac{\Omega^\omega_1}{|\Omega^\omega_1|^{1/3}}
\right)
=
\min_{\gamma\in\Gamma_\omega,\ |a|<\pi}
\mu_1^+
\left(
\frac{\Omega^\omega_{\gamma,a}}
{|\Omega^\omega_{\gamma,a}|^{1/3}}
\right).
\]
\end{corollary}

\begin{proof}
By Corollary~5.4, for each fixed \(\gamma\in\Gamma_\omega\), the
volume-normalized eigenvalue is uniquely minimized among shears at \(a=0\).
Thus it remains to minimize over the unsheared family
\[
\left\{
\frac{\Omega^\omega_\gamma}{|\Omega^\omega_\gamma|^{1/3}}
:
\gamma\in\Gamma_\omega
\right\}.
\]
For every \(\gamma\in\Gamma_\omega\), Theorem~4.9 gives
\[
\mu_1^+(\Omega^\omega_\gamma)=\lambda.
\]
Therefore the corresponding scale-invariant quantity is
\[
|\Omega^\omega_\gamma|^{2/3}
\mu_1^+(\Omega^\omega_\gamma)
=
\lambda|\Omega^\omega_\gamma|^{2/3}.
\]
By Corollary~A.3, the volume
\[
|\Omega^\omega_\gamma|
\]
attains its unique minimum at \(\gamma=1\) on \(\Gamma_\omega\). Hence
\[
|\Omega^\omega_\gamma|^{2/3}
\mu_1^+(\Omega^\omega_\gamma)
\]
is uniquely minimized at \(\gamma=1\). Combining this with the shear
optimality from Corollary~5.4, the full two-parameter minimum is uniquely
attained at
\[
(\gamma,a)=(1,0).
\]
\end{proof}

\section*{Concluding remarks}

We constructed an explicit three--dimensional family of double--pyramidal
domains
\[
\{\Omega^\omega_\gamma:\gamma\in\Gamma_\omega\},
\qquad
\omega=\frac{\Lambda}{\lambda}\ge 4,
\]
together with an explicitly constructed piecewise eigenfunction
\(u^\omega_\gamma\) satisfying
\[
-\mathcal{M}^+_{\lambda,\Lambda}(D^2 u^\omega_\gamma)
=
\lambda\,u^\omega_\gamma
\quad\text{in } \Omega^\omega_\gamma,
\qquad
u^\omega_\gamma=0
\quad\text{on } \partial\Omega^\omega_\gamma.
\]
This identifies the principal eigenpair on each unsheared shape
(Theorem~\ref{thm:explicit-eigenpair}).

We then quantified the spectral effect of the one--parameter shear \(C_a\).
For the sheared domains \(\Omega^\omega_{\gamma,a}=C_a(\Omega^\omega_\gamma)\),
we proved the lower bound
\[
\mu_1^+(\Omega^\omega_{\gamma,a})
\ge
\frac{\lambda\pi^2}{\pi^2-a^2},
\]
with equality in the admissible family precisely when \(a=0\)
(Theorem~\ref{thm:shear LB}).

The final minimization combines two structural mechanisms developed here:
\begin{enumerate}
\item exact \(C^1\) gluing across the seven patches;
\item the monotonicity of the volume functional
\[
V(\gamma)=|\Omega^\omega_\gamma|,
\]
obtained from the pairing of the cap and bridge contributions under the
duality \(\gamma\leftrightarrow\gamma^{-1}\).
\end{enumerate}
This volume monotonicity forces \(V(\gamma)\) to attain its unique minimum
at the symmetric value \(\gamma=1\). After normalization by scaling, this
yields our main outcome:
\emph{within the class of \(\gamma\)-deformations and affine shears considered
here, the volume-normalized principal eigenvalue is minimized uniquely by the
most symmetric, unsheared configuration} (Corollary~\ref{cor:min-at-1}).

\medskip

A conceptual point deserves emphasis. In the planar construction of
Birindelli--Leoni, the symmetric configuration is closely related to a
rotated square, and this permits a partial return to a separable Laplacian
geometry. In contrast, the present three-dimensional admissible domains
retain a genuinely seven-patch double-pyramidal structure. They are not
reduced, by an orthogonal change of variables, to a cubical separable
configuration compatible with the Pucci operator. The seven-patch
construction compensates for this loss of separability.

More fundamentally, the method relies on maintaining a \emph{fixed inertia
pattern} of the Hessian on each patch. The domain is engineered so that all
trigonometric profiles remain within intervals avoiding sign changes,
ensuring that the number of positive and negative second derivatives is
constant throughout each region. This allows the Pucci operator to be
evaluated patchwise without local switching of coefficients.

\paragraph{Outlook.}
The restriction \(\omega\ge4\) is intrinsic to the present seven-patch
construction. It ensures that the admissible interval
\[
\Gamma_\omega
=
\left[
\frac{2}{\sqrt\omega},
\frac{\sqrt\omega}{2}
\right]
\]
is nonempty and that the face caps attach to the full corresponding faces of
the central cube. The intermediate range
\[
1<\omega<4
\]
is not covered by the present argument and remains open. Treating this range
would require either a different admissible geometry or a modification of the
patch construction near the cap interfaces. The formal Laplacian case
\(\omega=1\) is degenerate from the viewpoint of the anisotropy parameter,
since it leaves only the symmetric value \(\gamma=1\), and should be regarded
separately.

The preceding mechanism also suggests both opportunities and limitations.
In higher dimensions, bridge-type regions may extend beyond half-period
windows of the trigonometric profiles, leading to local sign changes of
second derivatives and hence to variation in the Hessian inertia within a
single patch. Such inertia stratification would obstruct a direct extension
of the present argument. Whether an analogous phase-controlled patch
architecture exists in dimension \(n\ge4\) remains an open question.

Other natural directions include:
\begin{itemize}
\item extending the construction to \(\mathcal{M}^-_{\lambda,\Lambda}\), to Isaacs operators,
or to anisotropic Pucci classes;
\item classifying affine, or selected nonlinear, deformations that admit sharp
spectral lower bounds and determining all equality cases;
\item establishing quantitative stability estimates measuring the growth of
\(\mu_1^+(\Omega)\) with geometric distance from \(\Omega^\omega_1\);
\item investigating numerical optimization beyond the present domain class;
\item exploring connections with Pólya--Szegő-type symmetry principles
in fully nonlinear settings.
\end{itemize}

We hope that the explicit patch architecture and the
\(\gamma\leftrightarrow\gamma^{-1}\) pairing principle provide useful tools
for these problems.

\appendix\label{app:volume-derivative}

\section{Non-separability and volume monotonicity}\label{sec:pucci-check}
\subsection{Non-separability of Eigenfunctions on Cubes}
\label{appendix:nonsep}

In this appendix we provide the proof of Theorem~\ref{thm:nonsep},
which states that when \(\Lambda>\lambda\), the eigenfunctions of the Pucci
operator \(\mathcal{M}^+_{\lambda,\Lambda}\) on Cartesian domains such as
cubes are not separable in coordinate directions.

\begin{proof}
By the scaling law for the principal half-eigenvalue, it is enough to consider
\[
Q=\left(-\frac{\pi\sqrt n}{2},\frac{\pi\sqrt n}{2}\right)^n.
\]
Let \(u(x_1,\dots,x_n)\) be a positive eigenfunction of
\(\mathcal{M}^+_{\lambda,\Lambda}\) in \(Q\) associated with the principal
eigenvalue \(\mu>0\), that is,
\begin{equation}\label{eq:eigens}
\begin{cases}
-\mathcal{M}^+_{\lambda,\Lambda}(D^2u)=\mu u & \text{in } Q,\\
u>0 & \text{in } Q,\\
u=0 & \text{on } \partial Q.
\end{cases}
\end{equation}
Assume for contradiction that \(u\) is separable:
\[
u(x_1,\dots,x_n)=\prod_{j=1}^n f(x_j),
\]
where, by symmetry and normalization, \(f\) is smooth, even, positive, and
satisfies \(f(0)=1\). In particular, \(f'(0)=0\).

Evaluating at the origin, we have
\[
D^2u(0)=f''(0)I_n.
\]
Since \(u\) has a maximum at the origin, \(f''(0)<0\). Hence
\[
-\mathcal{M}^+_{\lambda,\Lambda}(D^2u(0))
=
-\lambda n f''(0)
=
\mu u(0)
=
\mu.
\]
Therefore
\[
f''(0)=-\frac{\mu}{n\lambda}<0.
\]

We first show that \(f''<0\) throughout the interval. Suppose, to the
contrary, that \(f''(x_0)=0\) for some
\[
x_0\in\left(-\frac{\pi\sqrt n}{2},\frac{\pi\sqrt n}{2}\right).
\]
Evaluating \eqref{eq:eigens} at the point \((x_0,0,\dots,0)\), the Hessian is
diagonal because \(f'(0)=0\). Its diagonal entries are
\[
0,\quad f(x_0)f''(0),\dots,f(x_0)f''(0).
\]
Thus
\[
\mu f(x_0)
=
-\lambda(n-1)f(x_0)f''(0).
\]
Since \(f(x_0)>0\), this gives
\[
\mu=-(n-1)\lambda f''(0),
\]
which contradicts
\[
\mu=-n\lambda f''(0).
\]
Hence \(f''<0\) everywhere.

Now evaluate \eqref{eq:eigens} at \((x,0,\dots,0)\). Again the Hessian is
diagonal, with entries
\[
f''(x),\quad f(x)f''(0),\dots,f(x)f''(0),
\]
and all these entries are negative. Therefore
\[
-\lambda\left(f''(x)+(n-1)f(x)f''(0)\right)=\mu f(x).
\]
Using
\[
f''(0)=-\frac{\mu}{n\lambda},
\]
we obtain
\[
f''(x)=-\frac{\mu}{n\lambda}f(x).
\]
Together with
\[
f\left(-\frac{\pi\sqrt n}{2}\right)
=
f\left(\frac{\pi\sqrt n}{2}\right)=0,
\qquad
f(0)=1,
\]
this gives
\[
f(x)=\cos\left(\frac{x}{\sqrt n}\right),
\qquad
\mu=\lambda.
\]
Thus
\[
u(x_1,\dots,x_n)
=
\prod_{i=1}^n
\cos\left(\frac{x_i}{\sqrt n}\right).
\]

Now consider the Hessian along the diagonal \((x,\dots,x)\). Put
\[
c=\cos\left(\frac{x}{\sqrt n}\right),
\qquad
s=\sin\left(\frac{x}{\sqrt n}\right).
\]
Then
\[
D^2u(x,\dots,x)
=
\frac{1}{n}
\begin{pmatrix}
-c^n & c^{n-2}s^2 & \cdots & c^{n-2}s^2\\
c^{n-2}s^2 & -c^n & \cdots & c^{n-2}s^2\\
\vdots & \vdots & \ddots & \vdots\\
c^{n-2}s^2 & c^{n-2}s^2 & \cdots & -c^n
\end{pmatrix}.
\]
This matrix has one eigenvalue
\[
e_1=
\frac{c^{\,n-2}}{n}\big((n-1)-nc^2\big)
\]
and \(n-1\) eigenvalues
\[
e_2=\cdots=e_n
=
-\frac{c^{\,n-2}}{n}.
\]
For \(x\) close to \(\pi\sqrt n/2\), we have \(c>0\) small, and hence
\[
e_1>0,
\qquad
e_2,\dots,e_n<0.
\]
Therefore
\[
-\mathcal{M}^+_{\lambda,\Lambda}(D^2u(x,\dots,x))
=
\Lambda c^n
-
\frac{n-1}{n}(\Lambda-\lambda)c^{n-2}.
\]
This differs from
\[
\lambda c^n
=
\mu u(x,\dots,x)
\]
unless \(\Lambda=\lambda\). This contradicts our assumption
\(\Lambda>\lambda\).

Hence the eigenfunction cannot be separable.
\end{proof}

\medskip
\noindent
This structural obstruction is the key motivation for adopting rhombus and double-pyramid geometries, 
where eigenfunctions admit piecewise separability after suitable affine transformations.

\subsection{Volume monotonicity}

We now prove that the volume of the unsheared domain is minimized at the
symmetric value \(\gamma=1\). It is enough to work in the first octant, since
the full domain is obtained by reflection across the coordinate planes.

Throughout this subsection we assume
\[
\omega\ge4,
\qquad
\gamma\in\Gamma_\omega
=
\left[
\frac{2}{\sqrt\omega},
\frac{\sqrt\omega}{2}
\right],
\qquad
\alpha=\frac{\pi}{2}.
\]
Let
\[
V^+(\gamma):=|\Omega^{\omega,+}_\gamma|,
\qquad
V(\gamma):=|\Omega^\omega_\gamma|.
\]
Then
\[
V(\gamma)=8V^+(\gamma).
\]
The central cube has volume \(\alpha^3\), independent of \(\gamma\). We write
the remaining first-octant volume as
\[
V^+(\gamma)
=
\alpha^3
+
Z_\gamma+2X_\gamma+2\mathcal{B}_\gamma+\mathcal{E}_\gamma.
\]
Here \(Z_\gamma\) denotes the \(Z\)-cap contribution,
\[
Z_\gamma
=
\sqrt\omega
\int_0^\alpha\int_0^\alpha
\arcsin\left(
\frac{\gamma(\cos x+\cos y)}{\sqrt\omega}
\right)\,dx\,dy,
\]
and, after renaming variables, the two \(X\)- and \(Y\)-caps have the same
volume
\[
X_\gamma
=
\sqrt\omega
\int_0^\alpha\int_0^\alpha
\arcsin\left(
\frac{\cos x+\gamma^{-1}\cos y}{\sqrt\omega}
\right)\,dx\,dy.
\]

For the \(ZX\)- and \(YZ\)-bridges, set
\[
T:=\{(\xi,\eta)\in[0,1]^2:\xi+\eta\le1\}.
\]
Using the changes of variables
\[
\xi=\sqrt\omega
\sin\left(\frac{x-\alpha}{\sqrt\omega}\right),
\qquad
\eta=\frac{\sqrt\omega}{\gamma}
\sin\left(\frac{z-\alpha}{\sqrt\omega}\right),
\]
we obtain the common bridge volume
\[
\mathcal{B}_\gamma
=
\gamma
\int_T
\frac{\arccos(\xi+\eta)}
{\sqrt{1-\xi^2/\omega}\sqrt{1-\gamma^2\eta^2/\omega}}
\,d\xi\,d\eta.
\]
Similarly, for the \(XY\)-bridge, using
\[
\xi=\gamma\sqrt\omega
\sin\left(\frac{x-\alpha}{\sqrt\omega}\right),
\qquad
\eta=\gamma\sqrt\omega
\sin\left(\frac{y-\alpha}{\sqrt\omega}\right),
\]
we get
\[
\mathcal{E}_\gamma
=
\gamma^{-2}
\int_T
\frac{\arccos(\xi+\eta)}
{\sqrt{1-\xi^2/(\gamma^2\omega)}
 \sqrt{1-\eta^2/(\gamma^2\omega)}}
\,d\xi\,d\eta.
\]

\begin{lemma}[Differentiation under the integral sign]\label{lem:diff-under-int}
The functions
\[
Z_\gamma,\qquad X_\gamma,\qquad \mathcal{B}_\gamma,\qquad
\mathcal{E}_\gamma
\]
are differentiable for \(\gamma\in\operatorname{int}\Gamma_\omega\), and
their derivatives are obtained by differentiating under the integral sign.
The same conclusion holds for the corresponding one-sided derivatives at the
endpoints of \(\Gamma_\omega\).
\end{lemma}

\begin{proof}
For \(\gamma\in\Gamma_\omega\), the admissibility conditions ensure that all
arcsine and arccosine expressions appearing above are well defined. On compact
subintervals of \(\operatorname{int}\Gamma_\omega\), the denominators in the
differentiated kernels are bounded away from zero, except possibly at boundary
points of the integration regions. The possible endpoint singularities are of
square-root type and are integrable. Hence the differentiated kernels are
locally dominated by integrable functions independent of \(\gamma\). The
claim follows from dominated convergence. The same argument gives the
one-sided derivatives at the endpoints.
\end{proof}

\begin{lemma}[Volume monotonicity]\label{lem:antisym}
Let
\[
V(\gamma)=|\Omega^\omega_\gamma|.
\]
If \(\omega>4\), then
\begin{equation}\label{eq:Vprime-antisym}
V'(\gamma)<0 \quad\text{for } \gamma<1,
\qquad
V'(1)=0,
\qquad
V'(\gamma)>0 \quad\text{for } \gamma>1.
\end{equation}
For \(\omega=4\), the admissible interval reduces to the single point
\(\Gamma_\omega=\{1\}\).
\end{lemma}

\begin{proof}
Since \(V(\gamma)=8V^+(\gamma)\), it is enough to prove the sign property for
\(V^+\). The case \(\omega=4\) is trivial because
\(\Gamma_\omega=\{1\}\). Hence we assume \(\omega>4\).

First consider the face-cap contribution
\[
C_\gamma:=Z_\gamma+2X_\gamma.
\]
By Lemma~\ref{lem:diff-under-int},
\[
Z'_\gamma
=
\int_0^\alpha\int_0^\alpha
\frac{\cos x+\cos y}
{\sqrt{
1-\dfrac{\gamma^2(\cos x+\cos y)^2}{\omega}
}}
\,dx\,dy.
\]
Also,
\[
X'_\gamma
=
-\gamma^{-2}
\int_0^\alpha\int_0^\alpha
\frac{\cos y}
{\sqrt{
1-\dfrac{(\cos x+\gamma^{-1}\cos y)^2}{\omega}
}}
\,dx\,dy.
\]
Using the symmetry of the square \([0,\alpha]^2\), we may write
\[
2X'_\gamma
=
-\gamma^{-2}
\int_0^\alpha\int_0^\alpha
\left[
\frac{\cos y}
{\sqrt{
1-\dfrac{(\cos x+\gamma^{-1}\cos y)^2}{\omega}
}}
+
\frac{\cos x}
{\sqrt{
1-\dfrac{(\cos y+\gamma^{-1}\cos x)^2}{\omega}
}}
\right]dx\,dy.
\]

Set
\[
c=\cos x,\qquad d=\cos y,\qquad S=c+d.
\]
If \(\gamma\ge1\), then
\[
c+\gamma^{-1}d\le S,
\qquad
d+\gamma^{-1}c\le S.
\]
Therefore
\[
\frac{d}
{\sqrt{1-(c+\gamma^{-1}d)^2/\omega}}
+
\frac{c}
{\sqrt{1-(d+\gamma^{-1}c)^2/\omega}}
\le
\frac{S}{\sqrt{1-S^2/\omega}}.
\]
Since \(\gamma^{-2}\le1\) and
\[
\frac{S}{\sqrt{1-\gamma^2S^2/\omega}}
\ge
\frac{S}{\sqrt{1-S^2/\omega}},
\]
we obtain
\[
C'_\gamma\ge0
\qquad\text{for }\gamma\ge1.
\]
The inequality is strict for \(\gamma>1\), except on a set of measure zero.
Hence
\[
C'_\gamma>0
\qquad\text{for }\gamma>1.
\]
Similarly, if \(0<\gamma<1\), then the above inequalities reverse and we get
\[
C'_\gamma<0.
\]
In particular,
\[
C'_1=0.
\]

We now treat the bridge contribution
\[
D_\gamma:=2\mathcal{B}_\gamma+\mathcal{E}_\gamma.
\]
Let
\[
h(\xi,\eta):=\arccos(\xi+\eta),
\]
and define
\[
A(t):=\sqrt{1-\frac{t^2}{\omega}},
\qquad
P_\gamma(t):=\sqrt{1-\frac{\gamma^2t^2}{\omega}},
\qquad
Q_\gamma(t):=\sqrt{1-\frac{t^2}{\gamma^2\omega}}.
\]
Then
\[
\mathcal{B}_\gamma
=
\gamma
\int_T
\frac{h(\xi,\eta)}
{A(\xi)P_\gamma(\eta)}
\,d\xi\,d\eta.
\]
Differentiating gives
\[
\mathcal{B}'_\gamma
=
\int_T
h(\xi,\eta)
\left[
\frac{1}{A(\xi)P_\gamma(\eta)}
+
\frac{\gamma^2\eta^2}
{\omega A(\xi)P_\gamma(\eta)^3}
\right]
\,d\xi\,d\eta.
\]
Also,
\[
\mathcal{E}_\gamma
=
\gamma^{-2}
\int_T
\frac{h(\xi,\eta)}
{Q_\gamma(\xi)Q_\gamma(\eta)}
\,d\xi\,d\eta,
\]
and hence
\[
\mathcal{E}'_\gamma
=
-\int_T
h(\xi,\eta)
\left[
\frac{2}{\gamma^3Q_\gamma(\xi)Q_\gamma(\eta)}
+
\frac{1}{\gamma^5\omega}
\left(
\frac{\xi^2}{Q_\gamma(\xi)^3Q_\gamma(\eta)}
+
\frac{\eta^2}{Q_\gamma(\xi)Q_\gamma(\eta)^3}
\right)
\right]
\,d\xi\,d\eta.
\]

Since \(T\) and \(h(\xi,\eta)\) are symmetric in \(\xi\) and \(\eta\), we may
write \(2\mathcal{B}'_\gamma\) as the integral of the sum of the above
integrand and its version with \(\xi\) and \(\eta\) interchanged.

Suppose first that \(\gamma\ge1\). Then, for \(0\le t\le1\),
\[
P_\gamma(t)\le A(t)\le Q_\gamma(t),
\qquad
\gamma^{-3}\le1,
\qquad
\gamma^{-5}\le\gamma^2.
\]
Consequently,
\[
\frac{1}{A(\xi)P_\gamma(\eta)}
\ge
\frac{1}{Q_\gamma(\xi)Q_\gamma(\eta)}
\ge
\frac{1}{\gamma^3Q_\gamma(\xi)Q_\gamma(\eta)},
\]
and
\[
\frac{\gamma^2\eta^2}
{\omega A(\xi)P_\gamma(\eta)^3}
\ge
\frac{\eta^2}
{\gamma^5\omega Q_\gamma(\xi)Q_\gamma(\eta)^3}.
\]
Applying the same estimates after interchanging \(\xi\) and \(\eta\), we get
\[
D'_\gamma=2\mathcal{B}'_\gamma+\mathcal{E}'_\gamma\ge0
\qquad\text{for }\gamma\ge1.
\]
The inequality is strict for \(\gamma>1\), except on a set of measure zero.
Thus
\[
D'_\gamma>0
\qquad\text{for }\gamma>1,
\qquad
D'_1=0.
\]

If \(0<\gamma<1\), then
\[
Q_\gamma(t)\le A(t)\le P_\gamma(t),
\qquad
\gamma^{-3}>1,
\qquad
\gamma^{-5}>\gamma^2.
\]
The preceding pointwise inequalities reverse, and we obtain
\[
D'_\gamma<0
\qquad\text{for }0<\gamma<1.
\]

Combining the cap and bridge contributions,
\[
(V^+)'(\gamma)=C'_\gamma+D'_\gamma.
\]
Therefore
\[
(V^+)'(\gamma)<0 \quad\text{for } \gamma<1,
\qquad
(V^+)'(1)=0,
\qquad
(V^+)'(\gamma)>0 \quad\text{for } \gamma>1.
\]
Since \(V(\gamma)=8V^+(\gamma)\), the same sign property holds for \(V\).
\end{proof}

\begin{corollary}[Appendix conclusion: volume minimized at \(\gamma=1\)]
\label{cor:appendix-min}
The function
\[
V(\gamma)=|\Omega^\omega_\gamma|
\]
attains its unique minimum at \(\gamma=1\) on \(\Gamma_\omega\).
\end{corollary}

\begin{proof}
If \(\omega=4\), then \(\Gamma_\omega=\{1\}\), and there is nothing to prove.
If \(\omega>4\), then by Lemma~\ref{lem:antisym}, \(V\) is strictly decreasing
on
\[
\left[\frac{2}{\sqrt\omega},1\right]
\]
and strictly increasing on
\[
\left[1,\frac{\sqrt\omega}{2}\right].
\]
Hence \(V\) attains its unique minimum at \(\gamma=1\).
\end{proof}

\section*{Acknowledgements}
Mohan Mallick gratefully acknowledges the financial support of the Anusandhan National Research Foundation (ANRF), Government of India, under grant no. ANRF/ARGM/2025/002309/MTR. Second author is supported by  National Board of Higher Mathematics grant no. 02011/36/2025/NBHM/RP/9466. 
\section*{Conflict of interest}
Authors have no conflict of interest.
\subsection*{Data availability} Data sharing is not applicable to this article as no datasets were
generated or analyzed during the current study.

\bibliographystyle{plain}
\bibliography{bib/references}

\end{document}